\documentclass[reqno]{amsart}

\usepackage{verbatim} 
\usepackage{amsmath}
\usepackage{amsfonts}
\usepackage{amssymb}
\usepackage{mathrsfs}
\usepackage{amsthm}
\usepackage{newlfont}
\usepackage{url}
\usepackage{filecontents}
\usepackage{fullpage,amsmath, amssymb,amscd}
\usepackage{latexsym, epsfig, color}
\usepackage[mathscr]{eucal}
\usepackage{xypic, graphicx}
\usepackage{amscd}
\usepackage[all, cmtip]{xy}

  \usepackage{pgfplots}
  \usepackage{pgfplotstable}
  \usetikzlibrary{patterns}

\usepackage[textwidth=50,textsize=tiny]{todonotes}
\newcommand\cyr{
 \renewcommand\rmdefault{wncyr}
 \renewcommand\sfdefault{wncyss}
 \renewcommand\encodingdefault{OT2}
\normalfont\selectfont} \DeclareTextFontCommand{\textcyr}{\cyr}

\newtheorem{theorem}{Theorem}
\newtheorem{lemma}[theorem]{Lemma}
\newtheorem{corollary}[theorem]{Corollary}
\newtheorem{proposition}[theorem]{Proposition}
\newtheorem{remark}[theorem]{Remark}

\newtheorem{conjecture}[theorem]{Conjecture}
\newtheorem{definition}[theorem]{Definition}

\def\Z{\mathbb Z}

\def\Q{\mathbb Q}
\def\R{\mathbb R}
\def\C{\mathbb C}
\def\F{\mathbb F}
\def\FF{\mathbb F}

\def\E{\mathbb E}
\def\G{\mathbb G}

\def\n{\noindent}

\def\Ker{\operatorname{Ker}}
\def\Im{\operatorname{Im}}

\def\det{\operatorname{det}}
\def\tr{\operatorname{tr}}

\def\End{\operatorname{End}}

\def\Gal{\operatorname{Gal}}

\def\Frob{\operatorname{Frob}}

\def\mod{\operatorname{mod}}
\def\dim{\operatorname{dim}}
\def\disc{\operatorname{disc}}
\def\car{\operatorname{char}}

\def\lcm{\operatorname{lcm}}

\def\GL{\operatorname{GL}}
\def\GSp{\operatorname{GSp}}
\def\USp{\operatorname{USp}}
\def\Sp{\operatorname{Sp}}
\def\USp{\operatorname{USp}}
\def\SL{\operatorname{SL}}

\def\li{\operatorname{li}}

\def\O{\operatorname{O}}
\def\o{\operatorname{o}}

\def\log{\operatorname{log}}

\def\ds{\displaystyle}
\def\Tr{\operatorname{Tr}}

\begin{document}

\title{
Arithmetic properties of the Frobenius traces defined by a rational abelian variety \\
{\footnotesize{with two appendices by J-P.~Serre}}}
\author[Alina Carmen Cojocaru]{Alina Carmen Cojocaru}
\address{Department of Mathematics, Statistics and Computer Science, University of Illinois at Chicago, 851 S Morgan St, 322
SEO, Chicago, IL, 60607, USA and 
Institute of Mathematics  ``Simion Stoilow'' of the Romanian Academy, 21 Calea Grivitei St, Bucharest, 010702,
Sector 1, Romania
}
\email{cojocaru@uic.edu}
\author[Rachel Davis]{Rachel Davis}
\address{ Department of Mathematics, Purdue University
150 N. University Street, West Lafayette, IN 47907, USA
}
\email{davis705@math.purdue.edu}
\author[Alice Silverberg]{Alice Silverberg}
\address{Department of Mathematics, University of California, Irvine, CA 92697-3875, USA}
\email{asilverb@uci.edu}
\author[Katherine E. Stange]{Katherine E. Stange}
\address{Department of Mathematics, University of Colorado, Boulder, Campus Box 395, Boulder, 80305, CO, USA}
\email{kstange@math.colorado.edu}

\thanks{
A.C. Cojocaru's  work on this material was partially supported by the National Science Foundation under agreement No. DMS-0747724,  by the European Research Council  under Starting Grant 258713, and by the Simons Collaborative Grant under Award No. 318454.
A.~Silverberg's work was partially supported by the National Science Foundation under agreement CNS-0831004. 
K.E.~Stange's work was partially supported by the National Science Foundation MSPRF 0802915, the Natural Sciences and Engineering Research Council of Canada PDF 373333, and sponsored by the National Security Agency under Grant H98230-14-1-0106.  The United States Government is authorized to reproduce and distribute reprints notwithstanding any copyright notation herein.
}

\begin{abstract}
Let $A$ be an abelian variety over $\Q$ of dimension $g$ such that the image of its  associated absolute Galois representation  $\rho_A$ is open in $\GSp_{2g}(\hat{\Z})$.
We investigate the arithmetic of the traces $a_{1, p}$ of the Frobenius at $p$ in $\Gal(\overline{\Q}/\Q)$ under $\rho_A$.
In particular, we obtain upper bounds for the counting function  $\#\{p \leq x: a_{1, p} = t\}$ and we prove 
an Erd{\"o}s-Kac type theorem for the number of prime factors of $a_{1, p}$. 
We also formulate a conjecture about the asymptotic behaviour of  $\#\{p \leq x: a_{1, p} = t\}$,
which generalizes a well-known conjecture of S.~Lang and H.~Trotter from 1976  about elliptic curves.
\end{abstract}

\maketitle

\section{Introduction}
\label{introsect}

Given an abelian variety $A/\Q$, its reductions $A_p/\F_p$ modulo primes  encode deep arithmetic global information.
A primary question related to these reductions concerns their $p$-Weil polynomials, in particular the coefficients of these 
polynomials. 

In the simplest case when $A$ has dimension 1, that is, when $A$ is an elliptic curve
over $\Q$,  
for each prime $p$ of good reduction the $p$-Weil polynomial is 
$
P_{A, p}(X) = X^2 - a_p X + p \in \Z[X],
$
where
$
a_p := p + 1 - |A_p(\F_p)|.
$
The coefficient $a_p$ satisfies the Weil bound
$|a_p| < 2 \sqrt{p}$
and is of major significance in number theory. For example, it appears as the $p$-th Fourier coefficient in the expansion of the weight 2 newform  associated to $A$.
The study of $a_p$ comes in several  flavours, some having led to well-known problems in arithmetic geometry, such as the Sato-Tate Conjecture from the 1960s (now a theorem) and  the Lang-Trotter Conjecture on Frobenius traces from the 1970s (still open).

Briefly, the Lang-Trotter Conjecture
\cite{LaTr} on the behaviour of $a_p$ predicts that for every elliptic curve $A/\Q$ and every integer $t \in \Z$,
if $\End_{\overline{\Q}}(A) \simeq \Z$ or 
$t \neq 0$, and if we write $N_A$ for the product of the primes of bad reduction for $A$, then 
either there are at most finitely many primes $p$ such that $a_{p} = t$, or 
  there exists a constant $c(A, t) > 0$ such that, as $x \rightarrow \infty$,
\begin{equation}\label{LT-dim1}
\pi_{A}(x, t) :=
\#
\left\{
p \leq x: p \nmid N_A, a_p = t
\right\}
\sim
c(A, t) \frac{\sqrt{x}}{\log x}.
\end{equation}
The constant $c(A, t)$ has a precise heuristic description  derived from  the Chebotarev Density Theorem,
combined with  the Sato-Tate Conjecture when  $\End_{\overline{\Q}}(A) \simeq \Z$  
and with  a prime distribution law arising from works of M. Deuring and E. Hecke  when  $\End_{\overline{\Q}}(A) \not\simeq \Z$.

While the Lang-Trotter Conjecture remains open, several remarkable related results have been proven. 
 When 
 $\End_{\overline{\Q}}(A) \not\simeq \Z$ (the CM case) and $t \neq 0$, upper bounds of the right order of magnitude can be proved using sieve methods. 
When $\End_{\overline{\Q}}(A) \simeq \Z$  and $t \neq 0$,
 weaker upper bounds, unconditional or conditional (upon the Generalized Riemann Hypothesis, GRH), can be proved using effective versions of the Chebotarev Density Theorem;
 such bounds   were first obtained by J-P.~Serre \cite[Th\'eor\`eme~20]{Se81}.
The currently best unconditional upper bound, $\pi_A(x, t) \ll_A \frac{x (\log \log x)^2}{(\log x)^2}$, was obtained by V.K. Murty \cite[Theorem~5.1]{Mu96} (see \cite{Wa} for an earlier result), while
 the currently best  upper bound under GRH,
$\pi_A(x, t) \ll_A \frac{x^{\frac{4}{5}}}{(\log x)^{\frac{1}{5}}}$,
was obtained by M.R. Murty, V.K. Murty $\&$ N.  Saradha  \cite[Theorem~4.2]{MuMuSa}
(for very recent improvements on the exponent of the $\log x$ factor, see \cite{Zy}).
 When $\End_{\overline{\Q}}(A) \simeq \Z$  and $t = 0$,
 stronger results are known; in particular, the unconditional bounds 
 $\frac{\log \log \log x}{(\log \log \log \log x)^{1 + \varepsilon}} \ll_{\varepsilon} \pi_A(x, 0) \ll x^{\frac{3}{4}}$  were obtained by 
  \'{E}. Fouvry $\&$ M.R. Murty  \cite[Theorem~1]{FoMu} and, respectively, by 
N.D. Elkies \cite{El} using, as a key tool, M. Deuring's characterization of supersingular primes \cite{De}.

Inspired by these works, the main goal of our paper is to investigate the arithmetic of the Frobenius traces of a generic higher dimensional abelian variety $A/\Q$; in particular: 

\noindent
(i) we will prove upper bounds for the generalization of the counting function $\pi_A(x, t)$ and  deduce results on the growth of the Frobenius traces; 

\noindent
(ii) we will determine  the normal order of the sequence defined by the  prime divisor function of the Frobenius traces, and,
 more generally, we  will prove an Erd{\"o}s-Kac type result for this sequence;
 
 \noindent
(iii)  under suitable hypotheses, we will formulate a generalization of (\ref{LT-dim1}).
 
\noindent
Our main results  mark  only the beginning of  such investigations in higher dimensions and we hope shall stimulate further research.

Our main setting and notation are as follows.
Let $A/\Q$ be a principally polarized abelian variety of dimension $g$.
Let $\overline{\Q}$ denote an algebraic closure of $\Q$ and let
$\End_{\overline{\Q}}(A)$ denote the endomorphism ring of $A$ over $\overline{\Q}$. 
 Let  $N_A$  be the product of primes of bad reduction for $A$. 

We denote by
$$
\rho_A : \Gal\left(\overline{\Q}/\Q\right) \longrightarrow \GSp_{2g}(\hat{\Z})
$$
the absolute Galois representation defined by the inverse limit of the 
representations
$$
\bar{\rho}_{A, m} : \Gal(\overline{\Q}/\Q) \longrightarrow \GSp_{2g}(\Z/m \Z)
$$
of $\Gal(\overline{\Q}/\Q)$ on the $m$-torsion $A[m] \subset A(\overline{\Q})$
for each integer $m \geq 1$. 
For each prime $\ell$ 
we denote by
$$
\rho_{A, \ell} : \Gal\left(\overline{\Q}/\Q\right) \longrightarrow \GSp_{2g}(\Z_{\ell})
$$
the $\ell$-adic representation,
i.e., the representation of $\Gal(\overline{\Q}/\Q)$ on the 
$\ell$-adic Tate module ${\displaystyle \lim_\leftarrow A[\ell^n}]$.

For each prime $p \nmid N_A$,   we consider the $p$-Weil polynomial
$
P_{A, p}(X)
$
  of $A$,
which is
uniquely determined by the property that 
\begin{equation}\label{p-Weil-char}
P_{A, p}(X) = \det \left(X I_{2 g} - \rho_{A, \ell} \left(   \Frob_p  \right)  \right) 
\end{equation}
for any prime $\ell \neq p$. In particular, we have
\begin{equation}\label{p-Weil-char-m}
P_{A, p}(X) \equiv \det \left(X I_{2 g} - \bar{\rho}_{A, m} \left(\Frob_p\right)  \right) (\mod m)
\end{equation}
for any integer $m$ coprime to  $p$. 
We write
$$
P_{A, p}(X) =
X^{2 g} + 
a_{1, p} X^{2 g - 1} + \ldots + a_{g, p} X^g 
+ 
p a_{g-1, p} X^{g-1} + \ldots + p^{g-1} a_{1, p} X + p^g  \in \Z[X],
$$
where the integers $a_{i, p}$, $1 \leq i \leq g-1$, are independent of $\ell$.

For any integer $t \in \Z$, we consider the function
\begin{equation*}
\pi_{A}(x, t) :=
\#
\left\{
p \leq x: p \nmid N_A, a_{1, p} =t
\right\}.
\end{equation*}

The reason we usually impose the restriction that our abelian varieties be
principally polarized is for ease of notation. When the abelian variety is
principally polarized, the image of the $\ell$-adic representation
$\rho_{A, \ell}$ lies in $\GSp_{2g}(\Z_{\ell})$. Without the restriction
on the polarization, the image lies in a group that can be defined by
replacing the matrix $J_{2g}$ of Section \ref{basicnotationsect} below with a 
matrix that has a more complicated description, and our results could be
modified accordingly; see, for example, Section 2.3 of \cite{Se86} for the 
group of symplectic similitudes in this general setting.

\begin{theorem}\label{main-thm1}
Let $A/\Q$ be a principally polarized abelian variety of dimension $g$ and let $t \in \Z$.
Assume that $\Im \rho_A$ is open in $\GSp_{2g}(\hat{\Z})$.
Define
$$
\alpha := \frac{1}{2 g^2 + g + 1}, 
\qquad
\beta := 
\left\{
\begin{array}{cc}
\frac{1}{3} & \text{if $g = 1$,}
\\
\frac{1}{2 g^2 - g + 3}& \text{if $g \geq 2$,}
\end{array}
\right.
\qquad
\gamma := 
\left\{
\begin{array}{ccc}
\frac{1}{2} & \text{if $g = 1$,}
\\
\frac{1}{8} & \text{if $g = 2$,}
\\
\frac{1}{2 g^2 - g + 1} \; & \text{if $g \geq 3$.}
\end{array}
\right.
$$
For any $\varepsilon > 0$ we have:
\begin{enumerate}
\item[(i1)]
unconditionally,
$$
\pi_A(x, t) \ll_{A,  \varepsilon}
 \frac{
 x
 }{
 (\log x)^{1 + \alpha - \varepsilon}
 };
$$
\item[(i2)]
under GRH,
$$
\pi_A(x, t) \ll_{A,  \varepsilon}
 x^{1 - \frac{\alpha}{2} + \varepsilon};
$$
\item[(ii)]
        if $t \neq \pm 2g$, 
then {\em (i1)} and {\em (i2)} hold with $\alpha$ replaced by $\beta$;
\item[(iii)]
if $t = 0$,
then  {\em (i1)} and {\em (i2)} hold with $\alpha$ replaced by $\gamma$.
\end{enumerate}
\end{theorem}
\noindent
Note that we will actually prove a more general result, stated as Theorem \ref{main-thm1'} in Section \ref{thmpfsect}, and that 
the case $g =1$  of Theorem \ref{main-thm1} is  \cite[Th\'eor\`eme~20, p.~189]{Se81}.

 An immediate consequence of Theorem \ref{main-thm1}
 concerns the non-lacunarity of the sequence $(a_{1, p})_p$:
 \begin{corollary}\label{cor-nonlac}
 We keep the setting and notation of Theorem \ref{main-thm1}. 
For any $\varepsilon > 0$ we have:
\begin{enumerate}
\item[(i)]
unconditionally,
$$
\#\left\{
p \leq x:  p \nmid N_A, |a_{1, p}| \geq (\log p)^{\alpha - \varepsilon}
\right\}
\sim
\pi(x);
$$
\item[(ii)]
under GRH,
$$
\#\left\{
p \leq x:  p \nmid N_A, |a_{1, p}| \geq p^{\frac{\alpha}{2} - \varepsilon}
\right\}
\sim
\pi(x).
$$
\end{enumerate}
 \end{corollary}

Recall that $\nu(n)$ denotes the number of distinct prime factors of 
a positive integer $n$
and that an arithmetic function $f(\cdot)$ is said to have normal order $F(\cdot)$
if for all $\varepsilon > 0$, then
$
(1 - \varepsilon) F(n) < f(n) < (1 + \varepsilon) F(n)
$
for  all but a zero density subset of positive integers $n$.
It is a classical result  of P.~Erd\"{o}s,
originating in work of G.H.~Hardy and S.~Ramanujan \cite{HaRa}, that 
$\nu(p-1)$ has normal order $\log \log p$. More generally, P.~Erd\"{o}s and M.~Kac \cite{ErKa} proved that $\nu(p-1)$ has a normal distribution.
 Variations of these results have also been obtained in arithmetic geometric contexts,
  including that of modular forms  \cite{MuMu}. We now prove such results in the context of abelian varieties:

\begin{theorem}\label{main-thm2}
Let $A/\Q$ be a principally polarized abelian variety of dimension $g$.
Assume that $\Im \rho_A$ is open in $\GSp_{2g}(\hat{\Z})$.
Under GRH we have that, for any $\tau \in \R$,
\begin{equation}\label{EK}
\ds\lim_{x \rightarrow \infty}
\frac{
\#\left\{
p \leq x: p \nmid N_A, a_{1, p} \neq 0, \nu(a_{1, p}) \leq \log \log p + \tau \sqrt{\log \log p}
\right\}
}{
\pi(x)
}
=
\frac{1}{\sqrt{2 \pi}}
\ds\int_{-\infty}^{\tau} e^{- \frac{t^2}{2}} \; d t.
\end{equation}
In particular, $\nu(a_{1, p})$ has normal order $\log \log p$.
\end{theorem}
\noindent
The case $g = 1$ not only recovers, but also generalizes the main theorem of \cite{MuMu} for weight 2 newforms that are not of CM type.

Finally, in Conjecture \ref{LT-dimg} below
we propose a generalization of (\ref{LT-dim1}) to the case of higher dimensional 
abelian varieties 
for which
 $\Im \rho_A$ is open in $\GSp_{2g}(\hat{\Z})$ and for which the following holds:
 
 \medskip
 
\noindent
{\bf Equidistribution Assumption}:
the normalized traces $\frac{ a_{1, p}}{\sqrt{p}}$
are equidistributed on $[-2g, 2g]$ with respect to the projection by the trace map
of the (normalized) Haar measure of  the unitary symplectic group $\USp(2g)$.

\medskip

The assumption that $\Im \rho_A$ is open in $\GSp_{2g}(\hat{\Z})$ gives rise to 
an integer $m_A \geq 1$ 
that is the smallest positive integer $m$ such that
\begin{equation*}\label{m_A}
\rho_A (\Gal(\overline{\Q}/\Q)) = \Pi^{-1}( \Im \bar{\rho}_{A, m}),
\end{equation*}
with 
$
\Pi : \GSp_{2 g}(\hat{\Z}) \longrightarrow \GSp_{2 g}(\Z/m \Z)
$
the natural projection.

The Equidistribution Assumption 
gives rise to a continuous function 
$
\Phi : [-1, 1] \longrightarrow [0, \infty),
$
nonzero at $0$, with the property that  for every interval $I \subseteq [-1, 1]$ we have
\begin{equation*}
\lim_{x \rightarrow \infty}
\frac{
\#\left\{p \leq x: p \nmid N_A, \frac{a_{1, p}}{2 g \sqrt{p}} \in I\right\}
}{
\pi(x)
}
=
\ds\int_I \Phi(t) \; d t.
\end{equation*}
We propose:

\begin{conjecture}\label{LT-dimg}
Let $A/\Q$ be a principally polarized abelian variety of dimension $g$
and
 let $t \in \Z$, $t \neq 0$.
Assume that  $\Im \rho_A$ is open in $\GSp_{2g}(\hat{\Z})$ and that
the Equidistribution Assumption holds.
Then,
as $x \rightarrow \infty$, 
$$
\pi_A(x, t)
\sim
c(A, t)
\frac{\sqrt{x}}{\log x},
$$
where
$$
c(A, t)
:=
\frac{\Phi(0)}{g}
\cdot
\frac{
m_{A, t}
|C(m_{A, t}, t)|
}{|\Im \bar{\rho}_{A, m_{A, t}} |}
\cdot
\ds\prod_{\ell \nmid m_A}
\frac{
\ell^{v_\ell(t)+1}
\;
|
\{M \in \GSp_{2g}(\Z/\ell^{v_\ell(t)+1} \Z) : \tr M \equiv t (\mod \ell^{v_\ell(t)+1})\}
|
}{ |\GSp_{2g}(\Z/\ell^{v_\ell(t)+1} \Z)| },
$$
the integers $v_{\ell}(t) \geq 0$ are defined by $\ell^{v_{\ell}(t)} | t$,  $ \ell^{v_{\ell}(t) + 1} \nmid t$,
and
$$
m_{A, t} :=
 m_A
 \ds\prod_{\ell | m_A} \ell^{v_{\ell}(t)},
$$
$$
C(m_{A, t}, t)
:=
\left\{
M \in \Im \bar{\rho}_{A, m_{A, t}}:
\tr M \equiv t (\mod m_{A, t})
\right\}.
$$
If $c(A, t) = 0$, we interpret the asymptotic as saying that there 
are at most finitely many primes $p$ such that $a_{1, p} = t$.
\end{conjecture}

For a discussion about the possible growth of $\pi_A(x, 0)$, see  Section 5.

\begin{remark}\label{open-remark}
{\emph{
The image of $\rho_A$ is open in $\GSp_{2g}(\hat{\Z})$
for a large class of abelian varieties.
Indeed, in 
\cite{Se86,Se86bis}
Serre showed that this holds 
whenever
$\End_{\overline{\Q}}(A) \simeq \Z$  and the dimension $g$ of $A$ is $1$, $2$, $6$, or an odd number.
An open image result also holds 
when $\End_{\overline{\Q}}(A) \simeq \Z$
and there exists a number field $K$ such that the N\'{e}ron model of $A/K$ over 
the ring of integers  of $K$ has a semistable fibre of toric dimension 1; 
see \cite{Ha}.
As pointed out in \cite[p.~704]{Ha},
for $g \geq 2$ the image of
$\rho_A$ is open in $\GSp_{2g}(\hat{\Z})$ for most abelian $g$-folds
that arise as Jacobians of hyperelliptic curves  defined by
$y^2 = f(x)$ with the degree $n$ of the monic polynomial  $f \in \Z[x]$ equal to $2 g + 1$ or $2 g + 2$.
Specifically, 
the hypotheses in Hall's Theorem are satisfied
if the Galois group of $f$ is $S_n$, 
or
if there exists a rational prime $p$ for which $f (\mod p)$ has $n-1$ distinct zeroes over an algebraic closure, one of which is a double zero; see E.~Kowalski's appendix to \cite{Ha} and Yu.~Zarhin's paper \cite{Za}.
}}
\end{remark}

\begin{remark}
{\emph{
When  $\Im \rho_A$ is open in $\GSp_{2g}(\hat{\Z})$,
the Equidistribution Assumption 
is a very special case of a general conjecture explained
in \S 13 of \cite[Conjecture~13.5]{Se94} that generalizes the Sato-Tate Conjecture.  See also \cite[pp.~173--174, 797--804, 906]{CoSe}  and \cite{Se66}.}}
\end{remark}

\begin{remark}
{\emph{
                Generalizations of the Lang-Trotter Conjecture (\ref{LT-dim1})
have been previously considered by other authors.
For example, in \cite{Mu99}, V.K. Murty addressed generalizations in the setting of modular forms,
while in \cite[pp.~421--423]{Ka09},  N. Katz  addressed  generalizations
in the setting of abelian varieties arising as Jacobians of genus $g$ curves.
 Our conjecture encompasses a generic class of abelian varieties $A$ and is  precise in terms of  both  the growth in $x$ and the constant depending  on $A$ and $t$. The potential vanishing of the constant $c(A, t)$  is an important open problem in itself.
 In \cite[p.~420]{Ka09}, for instance, Katz discusses a general mechanism that leads to congruence obstructions for realizing $a_{1, p} = t$.  
 We relegate this study to future work.
}}
\end{remark}

The paper is structured as follows.
In Section 2 we present some of the key results needed for proving Theorem \ref{main-thm1}, Corollary \ref{cor-nonlac}, and Theorem \ref{main-thm2}, and  for arguing towards Conjecture \ref{LT-dimg}. 
In Section 3 we prove Theorem \ref{main-thm1} and Corollary \ref{cor-nonlac} using the strategy of 
\cite[Sections 7-8]{Se81}  and  also with the help of the main result of 
Serre's Appendix \ref{SerreApp1} of this paper. 
In Section \ref{thmpfsect} we prove Theorem \ref{main-thm2} following a general strategy of \cite{Bi74}.
In Section 5 we provide our heuristic reasoning towards Conjecture \ref{LT-dimg} and  address some connections with existing works.
In Section 6 we provide computational data related to our theoretical investigations.  
J-P.~Serre supplied two appendices: the first gives a result on the dimension of conjugacy classes in symplectic groups, while the second gives properties of a certain density function for unitary symplectic groups.

\subsection*{Acknowledgments}
The authors thank  
Jeff Achter,
Alina Bucur, 
Hao Chen,
Francesc Fit\'e,
Nathan Jones,
Kiran Kedlaya, 
Emmanuel Kowalski, 
Antonella Perucca, 
Karl Rubin,
Jean-Pierre Serre,
Drew Sutherland, 
Cassie Williams
 and 
Jonathan Wise for valuable discussions related to this paper.  They are grateful for the hospitality  of
the Banff International Research Station, Alberta, Canada, 
and  the support of the organizers Chantal David, Matilde Lal\'in, and Michelle Manes of the conference
\emph{Women in Numbers 2} (2011), where this work was initiated. 
The last three authors thank the first for her leadership and hard work on this project.

\section{Generalities}

\subsection{Basic notation}
\label{basicnotationsect}

Along with the standard analytic notation 
$\O, \; \ll, \; \gg, \; \o, \; \sim, $  
$$\pi(x) := \#\{p \leq x: p \; \text{prime}\},$$  
$$\li x := \ds\int_{2}^x \frac{1}{\log t} \; d t,$$
we use
$p$ and $\ell$ 
to denote  rational primes;
we write $n | m^{\infty}$ to mean that all  the prime divisors of  $n$ occur among the prime divisors of $m$, possibly with higher multiplicities;
we write $n || m$ to mean that $n | m$, but $n \nmid m$;
 we write $v_{\ell}(n)$ for the valuation of $n$ at $\ell$.

For   a  commutative, unitary ring $R$ and a positive integer $g$,
we denote 
by $R^{\times}$ its group of units,
by  
$I_{g} \in M_{g}(R)$,
$I_{2g} \in M_{2g}(R)$ 
the identity matrices,
and by
$$
J_{2g}  :=   \begin{pmatrix} 0 & I_{g} \\  -I_g & 0 \end{pmatrix} \in M_{2 g}(R).
$$
We recall that the {\it{general symplectic group on $R$}} is defined by
$$
\GSp_{2g}(R) := 
\left\{ M \in \GL_{2g}(R) : M^t J_{2g} M = \mu J_{2g} \; \text{ for some} \;  \mu \in R^{\times} \right\},
$$
where $M^t$ denotes the transpose of $M$, while
$$
\Sp_{2g}(R) := 
\left\{ M \in \GL_{2g}(R) : M^t J_{2g} M = J_{2g} \right\}.
$$
We note that $\GSp_{2}(R) = \GL_2(R)$.
We recall that 
$\GSp_{2g}(R)$ 
has center  $\{\mu  I_{2g} :  \mu \in R^{\times} \}$ and that, as an algebraic group, it  has dimension $2g^2 + g+1$.

For $R = \C$, we recall that the {\it{unitary symplectic group}} is defined by
$$
\USp(2g) := 
\left\{
M \in \Sp_{2 g}(\C) : {\overline{M}}^t M = M {\overline{M}}^t = I_{2 g} 
\right\}.
$$

\subsection{The Chebotarev Density Theorem}

\subsubsection{Finite extensions of a number field}

Let $L/K$ be a finite Galois extension of number fields and let $G$ be its Galois group. Let $C$
be a non-empty subset of $G$ that is stable under conjugation. For any $x > 0$, let
$$
\pi_{C}(x, L/K) :=
\#
\{
\mathfrak{p} \; \text{a place of $K$, unramified in $L/K$}:
N_{K/\Q}(\mathfrak{p}) \leq x,
\Frob_{\mathfrak{p}} \subseteq C
\}.
$$
The Chebotarev Density Theorem states that
$$
\pi_{C}(x, L/K)
\sim \frac{|C|}{|G|} \pi(x).
$$
We will use the following conditional effective version of this theorem:

 \begin{theorem}(\cite{LaOd}; for this version see \cite[Th\'eor\`eme~4, p.~133]{Se81})\label{cheb2} 
 Keep the above setting and  notation.
 Assume GRH for the Dedekind zeta function of $L$. 
    Then there exists an absolute  constant
 $c > 0$ such that
 \begin{eqnarray*}
\left|
\pi_C(x, L/K) - \frac{|C|}{|G|} \pi(x)
\right|
\leq
c \frac{|C|}{|G|} x^{\frac{1}{2}}
\left(
\log |\disc (L/\Q)|
+
|L : \Q| \log x
\right).
\end{eqnarray*}
 \end{theorem}
 
 \noindent
In order to apply this theorem, the following variation of a result of Hensel \cite{He},   proved in  \cite{Se81}, is useful:

\begin{proposition}(\cite[Prop.~5, p.~129]{Se81}) \label{hensel}
Keep the above setting and notation.
Then
$$
\log \left|N_{K/\Q} (\disc (L/K))\right|
\leq
(|L : \Q| - |K : \Q|)
\left(
\ds\sum_{p \in {\cal{P}}(L/K)} \log p
\right)
+
|L : \Q| \log |L : K|,
$$
where
$$
{\cal{P}}(L/K) :=
\{ \text{primes 
$p$ :
there is a place $\mathfrak{p}$ of $K$, ramified in $L/K$, with $\mathfrak{p} | p $}
\}.
$$

\end{proposition}

\subsubsection{$\ell$-adic extensions of a number field}

In \cite{Se81}, Serre used the effective versions of the Chebotarev Density Theorem of Lagarias $\&$ Odlyzko  \cite{LaOd} to deduce 
upper bounds for $\pi_C(x, L/K)$ in the case of an $\ell$-adic Galois extension $L/K$ of a number field $K$. We recall his main results below.

Let $K$ be a number field. 
Let $\ell$ be a rational prime and $G$ a compact $\ell$-adic Lie group of dimension $D$. 
Denote by 
 $Z(G)$ the center of $G$.
Let $C \subseteq G$ be a non-empty closed subset of $G$
that is stable under conjugation.  
In \cite[Section 3]{Se81}  Serre explains what it means for the {\em Minkowski dimension} 
$\dim_{\mathcal{M}} C$ of $C$ to be $\le d$.
Let $L/K$ be an infinite Galois extension, with Galois group $G$. 
For any $x > 0$, let
$$
\pi_{C}(x, L/K) :=
\#
\{
\mathfrak{p} \; \text{a place of $K$, unramified in $L/K$}:
N_{K/\Q}(\mathfrak{p}) \leq x,
\Frob_{\mathfrak{p}} \subseteq C
\}.
$$
Following  \cite[p.151]{Se81}, we define
$$
\epsilon(x) : = 
\frac{
\log x
}{
(\log \log x)^2 (\log \log \log x)
}
\quad \text{and} \quad
\epsilon_R(x) : = 
\frac{
x^{\frac{1}{2}}
}{
(\log x)^2
}.
$$

\begin{theorem}(\cite[Th\'eor\`eme~10, p.~151]{Se81})\label{serre-cheb-thm10}
Keep the above setting and notation.
Let $0 \leq d < D$ be such that the Minkowski dimension of $C$ satisfies
$
\dim_{\cal{M}} C \leq d.
$
Define
$
\displaystyle{\alpha := \frac{D - d}{D}.}
$
\begin{enumerate}
\item[(i)]
Unconditionally, we have
$$
\pi_C(x, L/K) \ll_{K, L, C} \frac{\li x}{\epsilon(x)^{\alpha}}.
$$
In particular, for any $\varepsilon > 0$, we have 
$$
\pi_C(x, L/K) \ll_{K, L, C, \varepsilon}  \frac{x}{(\log x)^{1 + \alpha - \varepsilon}}.
$$
\item[(ii)]
Under GRH for Dedekind zeta functions, we have
$$
\pi_C(x, L/K) \ll_{K, L, C} \frac{\li x}{\epsilon_R(x)^{\alpha}}.
$$
In particular, for any $\varepsilon > 0$, we have
$$
\pi_C(x, L/K) \ll_{K, L, C, \varepsilon}  x^{1 - \frac{\alpha}{2} + \varepsilon}.
$$
\end{enumerate}
\end{theorem}

Serre obtains the following improvement in special cases:
\begin{theorem}({\cite[Th\'eor\`eme~12, p.~157]{Se81}}) \label{serre-cheb-thm12}
Keep the above setting and notation.
Let $0 \leq d < D$ be such that the Minkowski dimension of $C$ satisfies
$
\dim_{\cal{M}} C \leq d.
$
Define
$$
r_C := \inf_{M \in C} \dim \frac{G}{Z_G(M)},
$$
where $Z_G(M)$ denotes the centralizer of $M$ in $G$. 
Define
$$
\beta_C : = \frac{D - d}{D - \frac{r_C}{2}}.
$$
Then  $(i)$ and $(ii)$ of Theorem \ref{serre-cheb-thm10} hold with $\beta_C$ in place of $\alpha$. 
\end{theorem}

Note that $r_C \geq 0$, hence $\beta_C \geq \alpha$ and so Theorem \ref{serre-cheb-thm12} is Theorem \ref{serre-cheb-thm10} when 
$\beta_C = \alpha$. When $r_C \geq 1$, hence $\beta_C > \alpha$, Theorem \ref{serre-cheb-thm12} improves upon Theorem \ref{serre-cheb-thm10}. This happens when $C \cap Z(G) = \emptyset$.

\subsection{Abelian varieties}
\label{subsec:av}

Let $A/\Q$ be an
abelian variety of dimension $g$
and let $p$ be a prime of good reduction.
Recall that for any root $\pi \in \C$ of  $P_{A, p}(X)$ we have  $|\pi| = \sqrt{p}$,  
hence
\begin{equation}\label{bound-a1p}
|a_{1, p}| < 2 g \sqrt{p}.
\end{equation}

Property (\ref{p-Weil-char}) links the $p$-Weil polynomial $P_{A, p}(X)$ to the division fields of $A$,
in particular to the Galois representation defining $\rho_{A}$.

For  arbitrary integers $m \geq 1$ and $t$, 
we set
\begin{eqnarray*}
G(m) 
&:=&
 \Im \bar{\rho}_{A, m},
 \\
C(m, t)
& := &
\{M \in G(m): \tr M \equiv t (\mod m)\}.
\end{eqnarray*}
We recall that:

$\bullet$
by the N\'{e}ron-Ogg-Shafarevich criterion, 
\begin{equation}\label{ramification}
\text{
the extension $\Q(A[m])/\Q$ is unramified outside $m N_A$; 
}
\end{equation}

$\bullet$
by the injectivity of  the restriction of $\bar{\rho}_{A, m}$ to ${\Gal(\Q(A[m])/\Q)}$,
\begin{equation}\label{degree}
|G(m)| \leq |\GSp_{2 g} (\Z/m \Z)| \leq m^{2g^2 + g +1}.
\end{equation}

In many cases, the image of the representation $\rho_{A}$ is better understood.
For example, as already mentioned in Remark \ref{open-remark}  of Section \ref{introsect}, for several classes of abelian varieties $A/\Q$ with a trivial endomorphism ring, $\Im \rho_A$ is open in $\GSp_{2g} (\hat{\Z})$. In particular,
for such $A$ we have that:

$\bullet$
$ \Im \rho_{A, \ell}$ is open in $\GSp_{2 g}(\Z_{\ell})$ for all rational primes $\ell$;

$\bullet$
$G(\ell) \simeq \GSp_{2 g}(\Z/\ell \Z)$
for all but finitely many rational primes $\ell$.

Lemma \ref{open-torsion-cond} below gives  
further  consequences of the openness of $\Im \rho_A$  in $\GSp_{2g}(\hat{\Z})$.
To state the lemma, we introduce the following notation:
$$
F_t(m) := \frac{
m |C(m,t)|}
{|G(m)|}, \qquad
H_t(m) := 
\frac{
m 
|
\{M \in \GSp_{2g}(\Z/m \Z) : \tr M \equiv t (\mod m)\}
|
}{ |\GSp_{2g}(\Z/m \Z)| };
$$
for a sequence $(s_n)_n$, 
 $$\ds\lim_{m \; {\widetilde{\rightarrow}}  \;  \infty} s_m
 := 
 \ds\lim_{n \rightarrow \infty} s_{m_n}
 \; 
 \text{ with }\;
m_n := \ds\prod_{\ell \leq n} \ell^n.
$$

\begin{lemma}\label{open-torsion-cond}
Let $A/\Q$ be a principally polarized abelian variety of dimension $g$
such that $\Im \rho_A$ is open in $\GSp_{2g}(\hat{\Z})$.
\begin{enumerate}
\item[(i)]
There exists an integer $m \geq 1$ such that $\rho_A(\Gal(\overline{\Q}/\Q)) = \Pi^{-1}(G(m))$,
where we recall that
$$
\Pi : \GSp_{2 g}(\hat{\Z}) \longrightarrow \GSp_{2 g}(\Z/m \Z)
$$
is the natural projection. 
Denote by $m_A$ the least such integer.
\item[(ii)]
For all positive integers $m_1, m_2$ with
$
m_1 | m_A^{\infty} \; \; \text{and} \; \; (m_2, m_A) = 1,
$
we have
$$
G(m_1 m_2)  \simeq G(m_1) \times G(m_2) 
= G(m_1) \times \GSp_{2g}(\Z/m_2 \Z).
$$
\item[(iii)]
For all $t \in \Z$ we have
$$
\ds\prod_{ \ell}
H_t(\ell)
< \infty.
$$
In particular, if $t \neq 0$, then
$$
\ds\prod_{\ell \nmid m_A}
H_t
\left(
\ell^{v_{\ell}(t) + 1}
\right)
< \infty.
$$

\item[(iv)]
For all $t \in \Z$, $t \neq 0$, 
we have
\[
\ds\lim_{m \; {\widetilde{\rightarrow}}  \;  \infty}
F_t(m)
=
F_t
\left(
m_A \ds\prod_{\ell | m_A} \ell^{v_{\ell}(t)}
\right)
\cdot
\ds\prod_{\ell \nmid m_A}
H_t \left(\ell^{v_{\ell}(t) + 1} \right).
\]
\end{enumerate}
\end{lemma}
\begin{proof}
Parts (i) and (ii) are clear from  the  openness assumption on $\Im \rho_A$.
For part (iii), 
let $\ell \nmid m_A$ and $t$ be fixed.
First, we will show that 
\begin{equation}\label{quotient-first}
\frac{
\ell \; |C(\ell, t)|}{
|\GSp_{2g}(\Z/ \ell \Z)|
}
=
1 + \O\left(\frac{1}{\ell}\right).
\end{equation}

Recall that the multiplicator of $\GSp_{2g}(\Z /\ell \Z)$
is the character of  $\GSp_{2g}(\Z /\ell \Z)$ with kernel  $\Sp_{2g}(\Z /\ell \Z)$;
we denote it by $\text{mult}$.
Let $\car (M)$ denote the characteristic polynomial of a square matrix $M$.
For  $\gamma \in (\Z / \ell \Z)^{\times}$, define 
$$\GSp_{2g}(\Z /\ell \Z)^{\gamma} := \text{mult}^{-1}(\gamma),$$ 
$$C(\ell, t)^{\gamma} := C(\ell, t) \cap \GSp_{2g}(\Z / \ell \Z )^{\gamma},$$
 $$
{\mathcal{G}}(\ell)^{\gamma}
 :=
 \left\{
 \car (M) : 
 M \in \GSp_{2g}(\Z / \ell \Z )^{\gamma}
  \right\},
  $$ 
  $$
 {\mathcal{C}}(\ell, t)^{\gamma}
  :=
  \left\{
  M \in {\mathcal{G}}(\ell)^{\gamma} :
  \tr M = t
  \right\}.
  $$
  By \cite[Lemma 2.4, p.~631]{AcHo},
$$ 
\left(
 \frac{\ell}{\ell + 1} 
 \right)^{2g^2+g} 
 \frac{ |{\mathcal{C}}(\ell, t)^{\gamma}| }{|{\mathcal{G}}(\ell)^{\gamma}|} 
 \leq 
 \frac{ |C(\ell, t)^{\gamma}|  }{ |\Sp_{2g}(\Z / \ell \Z)| }
 \leq 
 \left( 
 \frac{\ell}{\ell - 1}
  \right)^{2g^2+g}
  \frac{|{\mathcal{C}}(\ell, t)^{\gamma}| }{|{\mathcal{G}}(\ell)^{\gamma}| }.
$$ 
Noting that 
$|{\mathcal{C}}(\ell, t)^{\gamma}|= \ell^{g-1} $ 
and 
$|{\mathcal{G}}(\ell)^{\gamma}| = \ell^g$, 
we deduce that
$$
\left( \frac{\ell}{\ell + 1} \right)^{2g^2+g} \cdot \frac{1 }{\ell }  
\leq 
\frac{ |C(\ell, t)^{\gamma}|  }{ |\Sp_{2g}(\Z / \ell \Z)| }  
\leq 
\left( \frac{\ell}{\ell - 1} \right)^{2g^2+g}  \cdot \frac{1 }{\ell }. 
$$
Combining the above inequalities for all $\gamma \in (\Z / \ell \Z)^{\times}$ 
and multiplying by $\ell$ gives
$$  
 \left( \frac{\ell}{\ell + 1} \right)^{2g^2+g}  
 \leq   
 \frac{\ell  \; |C(\ell, t)| }{ | \GSp_{2g}(\Z / \ell \Z ) | }  
 \leq 
 \left( \frac{\ell}{\ell - 1} \right)^{2g^2+g}.
 $$ 
This completes the proof of (\ref{quotient-first}).

Next  we will prove that
\begin{equation}\label{quotient-second}
\frac{
\ell \;  |C(\ell, t)|}{
|\GSp_{2g}(\Z/ \ell \Z)|
}
=
1 + \O\left(\frac{1}{\ell^2}\right).
\end{equation}
This ensures the convergence of the infinite product
$\ds\prod_{\ell} H_t(\ell)$,
proving 
(iii).

We first prove (\ref{quotient-second}) for $t \neq 0$.
For this, observe that for any  $t_1$, $t_2 \in \Z$  we have
$$
t_1 \equiv t_2 (\mod \ell) \; \Rightarrow \; C(\ell, t_1) = C(\ell, t_2)
$$
and
$$
t_1 \not\equiv 0 (\mod \ell), \;
t_2 \not\equiv 0 (\mod \ell) \;
\Rightarrow \; |C(\ell, t_1)| = |C(\ell, t_2)|.
$$
Indeed, the first assertion is trivial, while the second assertion follows by noting that,
if 
$t_1 \not\equiv 0 (\mod \ell)$
and
$t_2 \not\equiv 0 (\mod \ell)$,
then  
the endomorphism $\left[t_2 t_1^{-1}\right]$ of $\GSp_{2g}(\Z/\ell \Z)$ defined by multiplication by 
$t_2 t_1^{-1}$ is a bijection 
satisfying that
$\left[t_2 t_1^{-1}\right] (C(\ell, t_1)) = C(\ell, t_2)$.

From the above observations, 
$$
|\GSp_{2g}(\Z/\ell \Z)|
=
|C(\ell, 0)| + (\ell - 1) \;  |C(\ell, t)|.
$$  
It is now easy to show that (\ref{quotient-second})  follows
from this along with (\ref{quotient-first}) for $|C(\ell, 0)|$.

 Now we prove  (\ref{quotient-second}) for $t = 0$.
 When $g =1$, 
 a straightforward calculation gives that 
$$|C(\ell, 0)| = \ell^3 - \ell^2$$
 and so
$$
\frac{\ell |C(\ell, 0)|}{|\GL_2(\Z/\ell \Z)|}
=
\frac{\ell^3 (\ell - 1)}{\ell (\ell - 1) (\ell^2 - 1)}
=
\frac{\ell^2}{\ell^2 - 1} = 1 + \O\left( \frac{1}{\ell^2} \right).
$$
When $g \geq 2$, we proceed as follows.  
By \cite[Theorem~5.3, p.~170]{Ki}, 
        \begin{equation}
                \label{eqn:Kim}
        |C(\ell,t)| = g(\ell) + \left\{ \begin{array}{ll} - \ell^{-1}f(\ell) & \text{if} \;  t \neq 0, \\ \ell^{-1}(\ell-1)f(\ell) &  \text{if} \;  t = 0, \end{array} \right.
        \end{equation}
  for some explicit  polynomials
 $f(\ell)$ and $g(\ell)$ in $\ell$.
 Of relevance to us is that the degree $d_{g(\ell)}$ of the leading term 
 of $g(\ell)$ in $\ell$
  satisfies
 \begin{equation}\label{leading-g}
 d_{g(\ell)} = 2 g^2 + g,
 \end{equation}       
 and that the degree $d_{f(\ell)}$  of the leading term of $f(\ell)$ in $\ell$, while less explicit,
 can be shown to satisfy
\begin{equation}\label{leading-f}
 d_{f(\ell)} \leq \frac{3 g^2}{2} + \frac{g}{2} + 1.
 \end{equation}    
 
 Before justifying this bound, let us complete the proof of (\ref{quotient-second}) for $t = 0$, $g \geq 2$.    
From \eqref{eqn:Kim}, we see that for any $t \neq 0$,
$$
|C(\ell,0)| = f(\ell) + |C(\ell,t)|.
$$ 
Since we already know \eqref{quotient-second} for $t \neq 0$, it suffices to show that
$$
\frac{ \ell f(\ell) }{ |\GSp_{2g}(\Z/\ell\Z)| } = \O\left( \frac{1}{\ell^2} \right).
$$
This  follows from (\ref{leading-g}) and (\ref{leading-f}), as well as the assumption that $g \geq 2$; indeed,
 $$
 d_{g(\ell)} - d_{f(\ell)}  \geq \frac{g^2}{2} + \frac{g}{2} - 1 \geq 2.
$$
 Consequently, \eqref{quotient-second} holds 
  for $t=0$, $g \geq 2$.

Finally, let us justify (\ref{leading-f}).        
The expression for $f(\ell)$ is rather delicate; indeed, Kim showed that
 \begin{multline}
                \label{eqn:kim2}
                f(\ell) = \ell^{g^2-1}\sum_{b=0}^{\lfloor g/2 \rfloor} \ell^{b(b+1)} \prod_{m=0}^{2b-1} \frac{ (\ell^{g-m}-1) }{ (\ell^{2b-m}-1 )} \prod_{j=1}^b (\ell^{2j-1}-1)
        \sum_{k=1}^{\lfloor (g-2b+2)/2 \rfloor} \ell^k \\ \times \sum_{\alpha \in \FF_\ell^\times} K(\alpha)^{g-2b+2-2k}\sum_{\substack{j_1,\ldots,j_{k-1}\\2k-1 \le j_{k-1} \le \cdots \le j_1 \le g-2b+1}} \prod_{v=1}^{k-1} (\ell^{j_v - 2v}-1 ),
\end{multline}
where $K(\alpha)$ is the ordinary Kloosterman sum 
$$ 
K(\alpha) = K(\lambda; \alpha, 1) := \sum_{a \in \FF_q^\times} \lambda( a\alpha + a^{-1} )
$$
 for any non-trivial additive character $\lambda$ of $\FF_q$.

To find the leading term, we first focus on 
$$
\ds\sum_{\alpha \in \FF_\ell^\times} K(\alpha)^r
 $$
for an arbitrary integer $r  \geq 0$.       

When $r = 0$, the sum is simply $\ell-1$. 
When $r=1$, by Weil's  estimate on Kloosterman sums  $\left| K(\alpha) \right| \le 2\sqrt{\ell}$, we deduce that  
$\left|\ds\sum_{\alpha \in \FF_\ell^\times} K(\alpha)\right| \leq \ell^2$.
When $r \geq 2$, Kim remarks that
 \[
                \sum_{\alpha \in \FF_\ell^\times} K(\alpha)^r = 
                \ell^2 M_{r-1} - (\ell-1)^{r-1} + 2(-1)^{r-1}, 
   \] 
   where 
   $M_0 := 1$
and for any integer $s \geq 1$, 
$$
M_s 
:=
\left|\left\{
(\alpha_1, \ldots \alpha_s) \in (\FF_\ell^\times)^s :  
\alpha_1+\ldots+\alpha_s =1 \; \text{and}  \; \alpha_1^{-1} + \ldots + \alpha_s^{-1}=1
\right\}
\right|.
$$
Note that $M_1 = 1$ and that
for  $s \ge 2$, the first of the two conditions defining $M_s$  gives $\alpha_1$ linearly in terms of the other $\alpha_i$, while the second gives $\alpha_2$ as a root of a quadratic in the remaining terms.  
Thus, if $s \geq 2$, then $M_{s} \leq 2(\ell-1)^{s-2}$.  
It follows that when $r=2$, the sum
$\sum_{\alpha \in \FF_\ell^\times} K(\alpha)^r$
is bounded by an expression of leading degree at most $2$ in $\ell$ (by direct computation using $M_1$), and  when  $r \ge 3$, 
by an expression of leading degree at most $r-1$ in $\ell$.

        Using the above estimates, we now focus on the degree $d_{f(\ell)}$ of the leading term in  \eqref{eqn:kim2};  we deduce that
         \[
                d_{f(\ell)} \le \max\left\{ g^2 + 2bg - 2b^2 + b + k + kg - 2bk - k^2 + 1
                : 0 \le b \le \left\lfloor \frac{g}{2} \right\rfloor, 1 \le k \le \left\lfloor \frac{g-2b+2}{2} \right\rfloor \right\}.
        \]
The quadratic function above is maximized 
when $b = \lfloor \frac{g}{2} \rfloor$ and $k = \lfloor \frac{g-2b+2}{2} \rfloor =1$, with maximal value 
$\frac{3g^2}{2} + \frac{g}{2} + 1$;
 the bound
(\ref{leading-f})  follows. 
This proves \eqref{quotient-second}, and therefore  the first part of (iii).

To prove the second part of (iii), observe that $t \neq 0$ is divisible by at most finitely many primes, and so 
$\ds\prod_{\ell \nmid m_A}
H_t
\left(
\ell^{v_{\ell}(t) + 1}
\right)$
is a constant multiple of
$\ds\prod_{ \ell}
H_t(\ell)$,
hence finite by the first part of (iii).

Now we prove (iv). Fix an arbitrary $t \in \Z$ with $t \neq 0$.
For now, fix also a positive integer $m$ such that
$(m, m_A) = 1$ or $m_A |m$, and a prime divisor $\ell$ of $m$.
Write
$m = m_0\ell^{v_{\ell}(m)}$,
$t = t_0 \ell^{v_{\ell}(t)}$,
where $m_0, t_0 \in \Z$ satisfy $\ell \nmid m_0$, $\ell \nmid t_0$,
and note that $v_{\ell} (m) \geq 1$. For any $s \in \Z$ such that 
$s \equiv t (\mod m \ell^{v_{\ell}(t)})$,
we have $v_{\ell}(s) = v_{\ell}(t)$ since $v_{\ell}(m) \ge 1$. 
Therefore we may write
$
s = s_0\ell^{v_{\ell}(t)}
$
with $s_0 \in \Z$ and $\ell \nmid s_0$.  
By the Chinese Remainder Lemma, there exists $u \in \Z$ such that
$u \equiv t_0^{-1}s_0 \left(\mod \ell \right)$ 
and
$u \equiv 1 (\mod m_0)$,
hence such that
\begin{eqnarray}
        u \ t &\equiv& s \left(\mod \ell^{v_{\ell}(t) + 1} \right), \label{cinci} 
\\
u &\equiv& 1 (\mod m). \label{sase}
\end{eqnarray}

We have
\begin{equation}\label{sapte}
u I_{2 g} \in G(m),
\end{equation}
since if 
$(m, m_A) = 1$ then by (ii) we have 
$G(m) = \GSp_{2 g} (\Z/m \Z)$ and so 
$u I_{2 g} \in G(m)$, while 
if 
$m _A| m$ then 
$
u \equiv 1 (\mod m_A)
$ (by (\ref{sase}))
and 
$$
\left\{
M \in \GSp_{2 g} \left(\hat{\Z}\right):
M \equiv 1 (\mod m_A)
\right\}
\subseteq
\Im \rho_A
$$
(by the definition of $m_A$) and thus $u I_{2 g} \in G(m)$.

Using (\ref{cinci}) and (\ref{sapte}), we deduce that the multiplication by $u I_{2g}$ map
\begin{eqnarray*}
C\left(
m \ \ell^{v_{\ell}(t) + 1},  t
\right)
&\longrightarrow&
C\left(
m  \ \ell^{v_{\ell}(t) + 1}, s
\right)
\\
M &\mapsto& u I_{2 g} M
\end{eqnarray*}
is a bijection; in particular,
\begin{equation}\label{opt}
\left|
C\left(
m \ \ell^{v_{\ell}(t) + 1}, s
\right)
\right|
=
\left|
C\left(
m \ \ell^{v_{\ell}(t) + 1}, t
\right)
\right|.
\end{equation}

Now consider the natural projection
$$
\Pi: 
G\left(
m \ \ell^{v_{\ell}(t) + 1}
\right)
\longrightarrow
G\left(
m\ell^{v_{\ell}(t)}
\right)
$$
and observe that
\begin{equation}\label{noua}
\left|
\Pi^{-1}(I_{2g})
\right|
\cdot
\left|
G\left(
m\ell^{v_{\ell}(t)}
\right)
\right|
=
\left|
G\left(
m\ell^{v_{\ell}(t) + 1}
\right)
\right|.
\end{equation}

Letting
$$
S :=
\left\{
s \left(\mod m\ell^{v_{\ell}(t)}\right) :
s \equiv t \left(\mod m\ell^{v_{\ell}(t)}\right)
\right\},
$$
and using  (\ref{opt}), we obtain
\begin{eqnarray*}
\left|
\Pi^{-1}(I_{2g})
\right|
\cdot
\left|
C\left(
m\ell^{v_{\ell}(t)}, t
\right)
\right|
&=&
\left|
\Pi^{-1}
\left(
C\left(
m\ell^{v_{\ell}(t)}, t
\right)
\right)
\right|
\\
&=&
\left|
\ds\bigcup_{s \left(\mod m\ell^{v_{\ell}(t)}\right) \in S}
C\left(
m\ell^{v_{\ell}(t) + 1 }, s
\right)
\right|
\\
&=&
|S|
\cdot
\left|
C\left(
m\ell^{v_{\ell}(t) + 1 }, t
\right)
\right|
\\
&=&
\ell 
\cdot
\left|
C\left(
m\ell^{v_{\ell}(t) + 1 }, t
\right)
\right|,
\end{eqnarray*}
giving
\begin{equation}\label{zece}
\left|
\Pi^{-1}(I_{2g})
\right|
\cdot
\left|
C\left(
m\ell^{v_{\ell}(t)}, t
\right)
\right|
=
\ell
\cdot
\left|
C\left(
m\ell^{v_{\ell}(t) + 1}, t
\right)
\right|.
\end{equation}

Putting together (\ref{noua}) and (\ref{zece}), we  deduce
that for all positive integers $m$ such that
$(m, m_A) = 1$ or $m_A |m$, and for all primes $\ell\mid m$, we have
$$
F_t\left(
m\ell^{v_{\ell}(t) + 1 }
\right)
=
F_t\left( 
m\ell^{v_{\ell}(t)}
\right)
$$
and thus
$$
F_t\left(
m\ell^k
\right)
=
F_t\left(
m\ell^{v_{\ell}(t)}
\right) \quad \text{for all $k \geq v_{\ell}(t)$.}
$$
Therefore for all $d \mid m_A$ we have
\begin{equation}\label{trei}
F_t
\left(
d \ m_A \ds\prod_{\ell | m_A} \ell^{v_{\ell}(t)} 
\right)
=
F_t\left(
m_A \ds\prod_{\ell | m_A} \ell^{v_{\ell}(t)} 
\right)
\end{equation}
and for all $k \geq 1$ and all primes $\ell \nmid m_A$ we have
\begin{equation}\label{patru}
F_t\left(
\ell^{v_{\ell}(t) + k}
\right)
=
F_t\left(
\ell^{v_{\ell}(t) + 1}
\right).
\end{equation}

Now
for {\em any} positive integer $m$ consider its unique factorization
$$
m = m_1 \cdot m_2, \; \text{with} \; m_1 | m_A^{\infty} \;  \text{ and  $(m_2, m_A) = 1$.}
$$
By (ii),
$$
F_t(m) 
=
F_t(m_1)
\;
\ds\prod_{\ell | m_2}
H_t\left(
\ell^{v_{\ell}(m_2)}
\right).
$$

Using (\ref{trei}) for the second line below and (\ref{patru}) for the third line,
we have
\begin{eqnarray*}
\lim_{m \; \widetilde{\rightarrow} \; \infty} F_t(m)
&=&
\lim_{m \; \widetilde{\rightarrow} \; \infty} 
F_t(m_1) \;
\ds\prod_{\ell | m_2}
H_t\left(
\ell^{v_{\ell}(m_2)}
\right)
\\
&=&
F_t\left(
m_A \ds\prod_{\ell | m_A} \ell^{v_{\ell}(t)}
\right)
\; 
\cdot
\ds\lim_{x \rightarrow \infty}
\ds\prod_{
\ell < x
\atop{\ell \nmid m_A}
}
\ds\lim_{n \rightarrow \infty}
H_t\left(\ell^n\right)
\\
&=&
F_t\left(
m_A \ds\prod_{\ell | m_A} \ell^{v_{\ell}(t)}
\right)
\; 
\cdot
\ds\prod_{\ell \nmid m_A}
H_t\left(
\ell^{v_{\ell}(t) + 1}
\right),
\end{eqnarray*}
which gives (iv).
\end{proof}

\begin{remark}\label{remark-g2tracecounts}
{\emph{
As in the case $g =1$,  when $g = 2$ it is possible to derive closed formulae for the quotient
 $\frac{|C(\ell, t)|}{|\GSp_{4}(\Z/\ell \Z)|}$; indeed, we have
\[
        |\GSp_4(\Z/\ell\Z)| = \ell^4(\ell-1)(\ell^2-1)(\ell^4-1)
\]
and
\begin{equation}
        \label{eqn:g2clt}
        |C(\ell,t)|
        = \left\{ \begin{array}{ll}
                        \ell^5(\ell-1)(\ell^4-\ell-1) &  \; \text{if} \;  t = 0, \\
                        \ell^4(\ell^6-\ell^5-\ell^4+\ell+1) & \; \text{if} \;  t \neq 0. \\
                \end{array} \right. 
        \end{equation}
We sketch a proof of the latter using arguments from \cite{CaFoHuSu}; we leave it as an exercise to the reader to derive these formulae using the aforementioned results of \cite{Ki}.
Note that we will  use these formulae  in Remark \ref{example}.
\\
Define
        \[
                N_{\ell,t} := 
                \left|\left\{
                 (x,y,\delta) \in ((\Z/\ell\Z)^\times)^3 : y \neq -\delta, \left(x + \frac{y}{x}\right) \left(1 + \frac{\delta}{y}\right) = t 
                 \right\}
                 \right|.
        \]
It follows from the proof of \cite[Theorem~12]{CaFoHuSu} that 
$$
\left|\left\{
M \in \GSp_{4}(\Z/\ell\Z): \tr M \equiv t  (\mod \ell)
\right\}
\right|
$$
equals
        \[
                \ell^4\left( (\ell-1)^2(\ell-2) + N_{\ell,t} \right)
                + \ell^4(\ell-1)(\ell^2-1)^2
                + \ell^5(\ell-1)^2(\ell^3-\ell-1)
                + \left\{ \begin{array}{cc}
                                (\ell^7 - \ell^4)(\ell-1) & \; \text{if} \;  t=0, \\
                                0 & \; \text{if} \; t \neq 0.
                        \end{array} \right. 
                \]
  We will now show that
        \[
                N_{\ell,t} = \left\{ \begin{array}{cc}
                                (\ell-1)(\ell-2) & \; \text{if} \;  t = 0, \\
                                (\ell-2)^2 & \; \text{if} \;  t \neq 0, \\
                        \end{array} \right. 
                \]
which in turns confirms  (\ref{eqn:g2clt}).
}}

{\emph{
      Note that
       \[ 
       \left| \left\{
       (x,y,\delta) \in ((\Z/\ell\Z)^\times)^3 : y \neq -\delta \right\}\right|  = (\ell-1)^2(\ell-2)
                 \]
 and
  $$
  N_{\ell,0}
  =
  \left| \left\{
 (x,y,\delta) \in ((\Z/\ell\Z)^\times)^3 : y \neq -\delta, \left(x + \frac{y}{x}\right) \left(1 + \frac{\delta}{y}\right) = 0 
  \right\}\right|
   = (\ell-1)(\ell-2),
  $$
  since, for  any given $x$, the defining conditions of these sets determines $y$  uniquely, provided that $\delta \neq -y$.  
 Putting the two together, we obtain that
 \[  
       \left| \left\{
                 (x,y,\delta) \in ((\Z/\ell\Z)^\times)^3 : y \neq -\delta, \left(x + \frac{y}{x}\right) \left(1 + \frac{\delta}{y}\right) \ne 0 
         \right\}\right| = (\ell-1)(\ell-2)^2.
                 \] 
 Dividing by $\ell-1$, we  deduce that 
  $N_{\ell,t} = (\ell-2)^2$
  for any fixed nonzero $t$;
 this completes the proof of (\ref{eqn:g2clt}).
}} 
\end{remark}

\section{Proof of Theorem \ref{main-thm1}}

For a prime $\ell$ and an integer $t$, define:
\begin{eqnarray*}
\G_{\ell} &:=& \GSp_{2g}(\Z_{\ell});
\\
P\G_{\ell} &:=& \G_{\ell} / Z(\G_{\ell});
\\
\Pi: \G_{\ell} & \longrightarrow  & P\G_{\ell} \; \text{the canonical projection};
\\
G_{\ell} & := & \Im \rho_{A, \ell};
\\
G'_{\ell} &:=& \Pi(G_{\ell});
\\
\C_{\ell}(t) &:=& \left\{M \in \G_{\ell} : \tr M = t\right\};
\\
C_{\ell}(t) &:=& \left\{M \in G_{\ell} : \tr M = t\right\};
\\
C'_{\ell}(0) &:=& \Pi \left(C_{\ell}(0)\right);
\\
r_{C_{\ell}(t)} & :=& \ds\inf_{M \in C_{\ell}(t)} \dim \frac{G_{\ell}}{Z_{G_{\ell}}(M)};
\\
r_{C'_{\ell}(0)} & :=& \ds\inf_{M \in C'_{\ell}(0)} \dim \frac{G'_{\ell}}{Z_{G'_{\ell}}(M)}.
\end{eqnarray*}

We will deduce Theorem \ref{main-thm1} from the following more general result:
\begin{theorem}\label{main-thm1'}
Let $A/\Q$ be a principally polarized abelian variety of dimension $g$ and let $t \in \Z$.
\begin{enumerate}
\item[(i)]
Assume that there exists a prime $\ell$ such that:
\begin{enumerate}
\item
 $G_\ell$ is open in $\G_\ell$;
 \item
 $\exists \; 0 \leq d < \dim \G_{\ell} \; \text{such that} \; \dim_{{\cal{M}}} C_{\ell}(t) \leq d.$
\end{enumerate}
Define
$$
\alpha :=
\frac{
\dim \G_{\ell} - d
}{
\dim \G_{\ell}
}. 
$$
Then
for any $\varepsilon > 0$ we have:
\begin{enumerate}
\item[(i1)]
unconditionally,
\begin{equation}\label{piA-uncond}
\pi_A(x, t) \ll_{A, \ell, \varepsilon}
 \frac{
 x
 }{
 (\log x)^{1 + \alpha - \varepsilon}
 };
\end{equation}
\item[(i2)]
under GRH,
\begin{equation}\label{piA-GRH}
\pi_A(x, t) \ll_{A, \ell, \varepsilon}
 x^{1 - \frac{\alpha}{2} + \varepsilon}.
 \end{equation}
\end{enumerate}
\item[(ii)]
If $t \neq \pm 2g$, assume that there exists a prime $\ell$ such that:
\begin{enumerate}
\item
 $G_\ell$ is open in $\G_\ell$;
 \item
 $\exists \; 0 \leq d < \dim \G_{\ell} \; \text{such that} \; \dim_{{\cal{M}}} C_{\ell}(t) \leq d;$
\item
$v_\ell(\frac{t}{2g}) \neq 0$.
\end{enumerate}
Define
$$
\beta := 
\frac{
\dim \G_{\ell} - d
}{
\dim \G_{\ell} - \frac{r_{C_{\ell}(t)}}{2}
}.
$$
Then $r_{C_{\ell}(t)} > 0$ and the equations (\ref{piA-uncond}) and (\ref{piA-GRH}) hold with $\alpha$ replaced by $\beta$.
\item[(iii)]
If $t=0$,
assume that there exists a prime $\ell$ such that:
\begin{enumerate}
\item
$G_\ell$ is open in $\G_\ell$;
\item
$ \exists \; 0 \leq d < \dim P\G_{\ell} \; \text{such that} \; \dim_{{\cal{M}}} C'_{\ell}(0) \leq d.$
\end{enumerate}
Define
$$
\gamma := 
\frac{
\dim \G_{\ell} -1 - d
}{
\dim \G_{\ell} - 1 - \frac{r_{C'_{\ell}(0)} }{2}
}.
$$
Then $r_{C'_{\ell}(0)} > 0$ and the equations (\ref{piA-uncond}) and (\ref{piA-GRH}) hold with $\alpha$ replaced by $\gamma$.
\end{enumerate}
\end{theorem}

\begin{proof}
Throughout the proof we let $x > 0$, to be thought of as approaching $\infty$. 

(i) Observe that, by (\ref{p-Weil-char}), for any rational prime $\ell$ we have
$$
\pi_A(x, t) \leq \pi_{C_{\ell}(t)}(x, L/\Q),
$$
where
$$
L := \overline{\Q}^{\Ker \rho_{A, \ell}}.
$$
It remains to estimate $\pi_{C_{\ell}(t)}(x, L/\Q)$, which we do by 
following the method of \cite[Section 8]{Se81}.

We choose $\ell$ as in the hypothesis of (i). Note that since $G_\ell$ is open in $\G_\ell$, we have
$\dim G_{\ell} = \dim \G_{\ell}$. 
We   apply Theorem \ref{serre-cheb-thm10} to the extension 
$L/\Q$ and the conjugacy set $C_{\ell}(t)$ with 
$D := \dim \G_{\ell}$.

\noindent
(ii) If $t \neq \pm 2g$,  we choose $\ell$ as in the hypothesis of (ii).
As before, $\dim G_{\ell} = \dim \G_{\ell}$. Moreover, 
$$\C_{\ell}(t) \cap Z(\G_{\ell}) = \emptyset,$$ 
for, otherwise,  recalling that
$
Z(\G_\ell) 
=
\{\mu I_{2 g}: \mu \in \Z_{\ell}^{\times}\},
$
we would have that the $\ell$-adic valuation of $\frac{t}{2g}$ satisfies $v_{\ell}\left(\frac{t}{2 g}\right) = 0$,
a contradiction.  
 
In particular, for any $M \in \C_{\ell}(t)$, 
\begin{equation}\label{cent-ii}
Z_{\G_\ell}(M) \subsetneq \G_\ell.
\end{equation}
Centralizers are closed subgroups, hence Lie subgroups, and $Z_{\G_\ell}(M)$ has a well-defined dimension.  Since $\GSp_{2g}$ is connected as an algebraic group, (\ref{cent-ii})
implies that 
$$\dim Z_{\G_\ell}(M) < \dim \G_\ell = \dim G_\ell.$$  

If $M \in C_{\ell}(t)$, then 
$\dim Z_{G_{\ell}}(M) \leq \dim Z_{\G_{\ell}}(M)$
and, by the above,
\begin{eqnarray*}
\dim \frac{G_{\ell}}{Z_{G_{\ell}}(M)}
\geq
\dim \G_{\ell} - \dim Z_{\G_{\ell}}(M) \geq 1.
\end{eqnarray*}
Therefore we can improve upon the result of (i) by applying
Theorem \ref{serre-cheb-thm12} to the extension 
$L/\Q$ and the conjugacy set $C_{\ell}(t)$ with 
$D := \dim \G_{\ell}$. 

\medskip

\noindent
(iii) 
If $t = 0$,  we choose $\ell$ as in the hypothesis of (iii)
and 
with $\hat{\rho}_{A, \ell} := \Pi \circ \rho_{A, \ell}$ we consider
 $$
 L' := \overline{\Q}^{\Ker \hat{\rho}_{A, \ell}},
 $$
 a Galois extension of $\Q$ with Galois group $G'_{\ell}$.
Observing that
 $$
\pi_A(x, 0) \leq \pi_{C'_{\ell}(0)}(x, L'/\Q),
$$
it remains to estimate the right hand side.
 
Since $G_\ell$ is open in $\G_\ell$, we have that $G'_\ell$ is open in $P\G_\ell$ and so
 $\dim G'_{\ell} = \dim P\G_\ell = \dim \G_{\ell} - 1$. 
  Moreover, 
 since $Z(P\G_\ell) = \{I_{2g}\}$, we have
 $$\Pi(\C_{\ell}(0)) \cap Z(P\G_{\ell}) = \emptyset.$$
 
 In particular, as in the proof of part (ii),  for any  $M \in \C_{\ell}(0)$, 
 \begin{equation*}\label{cent-iii}
 Z_{P\G_{\ell}}(\Pi(M)) \subsetneq P\G_{\ell},
 \end{equation*}
thus
$$
\dim Z_{P\G_\ell}(\Pi(M)) 
<
\dim \G_\ell - 1.
$$

If  $M  \in C_{\ell}(0)$,
then 
$\dim Z_{G'_{\ell}}(\Pi(M)) \leq \dim Z_{P\G_{\ell}}(\Pi(M))$
and, by the above,
\begin{eqnarray*}
        \dim \frac{G'_{\ell}}{Z_{G'_{\ell}}(\Pi(M))}
\geq
\dim P\G_{\ell} - \dim Z_{P\G_{\ell}}(\Pi(M)) \geq 1.
\end{eqnarray*}
Therefore we can improve upon the result of (i) by applying
Theorem \ref{serre-cheb-thm12} to the extension 
$L'/\Q$ and the conjugacy set $C'_{\ell}(0)$ with 
$D := \dim \G_{\ell} - 1$. 
\end{proof}

\begin{proof}[Proof of Theorem \ref{main-thm1}.]
In our setting, by the openness assumption on $\Im \rho_A$,
hypothesis (a)  of Theorem \ref{main-thm1'} holds for any prime $\ell$.       
It remains to verify hypothesis (b)  and to compute the values of $\alpha$, $\beta$ and $\gamma$.

To verify hypothesis (b) of either parts (i)  or (ii), observe that 
$\C_{\ell}(t)$ is  a closed subvariety of the algebraic group  $\GSp_{2g}$ and so
 $C_\ell(t)$ has a well-defined dimension strictly smaller than $\dim \G_\ell$.
 The bound applies to the Minkowski dimension $\dim_{\mathcal{M}}C_\ell(t)$ also, by \cite[Theorem 8]{Se81}.  Part (b) follows with $d := \dim \G_\ell - 1$.

 To verify hypothesis (b) of  part (iii), observe that 
$\Pi(\C_{\ell}(0))$ is  a closed subvariety of the algebraic group  $P\G_{\ell}$ and so
$C'_{\ell}(0)$ has a well-defined dimension strictly smaller than $\dim P\G_\ell$.  The bound applies to the Minkowski dimension $\dim_{\mathcal{M}}C'_\ell(0)$ also by \cite[Theorem 8]{Se81}.
 Part (b) follows with $d := \dim \G_\ell - 2$.

Recalling that 
$\dim \GSp_{2g} = 2g^2 + g + 1$, we see that $\alpha = \frac{1}{2 g^2 + g + 1}$.

If $g=1$, then
$r_{C_{\ell}}(t)$ and  $r_{C'_{\ell}}(0)$ are calculated as in \cite[pp.~189--190]{Se81}, giving rise to $\beta = \frac{1}{3}$ and $\gamma = \frac{1}{2}$.
 If $g \geq 2$, then
$r_{C_{\ell}}(t)$ and  $r_{C'_{\ell}}(0)$  are estimated using
 Serre's Theorem in Appendix \ref{SerreApp1}.
  Indeed, by this theorem and Remark 5 that follows its statement, for $M \in C_{\ell}(t)$ with $t$ as in (ii) we have
 $$
 r_{C_{\ell}(t)} =  
\ds\inf_{M \in C_{\ell}(t)} \dim \frac{G_{\ell}}{Z_{G_{\ell}}(M)}
=
 \ds\inf_{M \in \C_{\ell}(t)} \dim \frac{\G_{\ell}}{Z_{\G_{\ell}}(M)}
 \geq 4 g-4,
 $$
 which gives
 $$
 \beta \geq \frac{1}{2 g^2 - g + 3}.
 $$
 To improve upon this bound when $t = 0$, we focus on  estimating $\gamma$
 and use
\begin{equation}
\label{ZPGclaim}
 \dim Z_{P \G_{\ell}} \left(\Pi(M)\right) = \dim Z_{\G_{\ell}} (M) - 1.
\end{equation}

 If $g =2$, we use \eqref{ZPGclaim} and once again the first part of the Theorem in Appendix \ref{SerreApp1} to deduce
 $$
 \gamma 
 \geq 
 \frac{1}{
 (2 g^2 + g + 1)
 -1
 - \frac{4 g - 4}{2}
 }
 =
 \frac{1}{8}.
 $$
 If $g \geq 3$, we use \eqref{ZPGclaim} and the last part of the Theorem in  Appendix \ref{SerreApp1} to deduce
$$
\gamma
\geq
\frac{
1
}{
(2 g^2 + g + 1) - 1 - \frac{4 g - 2}{2}
}
=
\frac{1}{2 g^2 - g + 1}.
$$
 This completes the proof of Theorem \ref{main-thm1}.
 \end{proof}
 
\begin{proof}[Proof of Corollary \ref{cor-nonlac}.]
The proof of Corollary \ref{cor-nonlac} is deduced easily from 
part (i) of Theorem \ref{main-thm1} and the Prime Number Theorem, as follows.
Unconditionally,
\begin{eqnarray*}
\pi(x)
&=&
\#\{p \leq x: p | N_A\}
+
\#\left\{
p \leq x: p \nmid N_A,  |a_{1, p}| \geq (\log p)^{\alpha - \varepsilon}
\right\}
+
\#\left\{
p \leq x: p \nmid N_A,  |a_{1, p}| < (\log p)^{\alpha - \varepsilon}
\right\}
\\
&=&
\#\left\{
p \leq x: p \nmid N_A,  |a_{1, p}| \geq (\log p)^{\alpha - \varepsilon}
\right\}
+
\O_{A}(1)
+
\O\left(
\ds\sum_{
t \in \Z
\atop{|t| < (\log x)^{\alpha - \varepsilon}}
}
\pi_A(x, t)
\right)
\\
&=&
\#\left\{
p \leq x: p \nmid N_A,  |a_{1, p}| \geq (\log p)^{\alpha - \varepsilon}
\right\}
+
\O_{A}(1)
+
\O_{A, \varepsilon} \left(
\frac{
x
}{
(\log x)^{1 + \alpha - \frac{\varepsilon}{2}}
}
\cdot
 (\log x)^{\alpha - \varepsilon}
\right)
\\
&=&
\#\left\{
p \leq x: p \nmid N_A,  |a_{1, p}| \geq (\log p)^{\alpha - \varepsilon}
\right\}
+
\o\left(\pi(x)\right).
\end{eqnarray*}
Under GRH,
\begin{eqnarray*}
\pi(x)
&=&
\#\{p \leq x: p | N_A\}
+
\#\left\{
p \leq x: p \nmid N_A,  |a_{1, p}| \geq p^{\frac{\alpha}{2} - \varepsilon}
\right\}
+
\#\left\{
p \leq x: p \nmid N_A,  |a_{1, p}| <  p^{\frac{\alpha}{2} - \varepsilon}
\right\}
\\
&=&
\#\left\{
p \leq x: p \nmid N_A,  |a_{1, p}| \geq p^{\frac{\alpha}{2} - \varepsilon}
\right\}
+
\O_{A}(1)
+
\O\left(
\ds\sum_{
t \in \Z
\atop{|t| < x^{\frac{\alpha}{2} - \varepsilon}}
}
\pi_A(x, t)
\right)
\\
&=&
\#\left\{
p \leq x: p \nmid N_A,  |a_{1, p}| \geq p^{\frac{\alpha}{2} - \varepsilon}
\right\}
+
\O_{A}(1)
+
\O_{A, \varepsilon} \left(
x^{1 - \frac{\alpha}{2} + \frac{\varepsilon}{2}} 
\cdot
x^{\frac{\alpha}{2} - \varepsilon}
\right)
\\
&=&
\#\left\{
p \leq x: p \nmid N_A,  |a_{1, p}| \geq p^{\frac{\alpha}{2} - \varepsilon}
\right\}
+
\o\left(\pi(x)\right).
\end{eqnarray*}
Note that the uniformity in $t$ of the bounds for $\pi_A(x, t)$ provided by Theorem \ref{main-thm1} was crucial in the above estimates.
\end{proof}

\section{Proof of Theorem \ref{main-thm2}}
\label{thmpfsect}

Let $A/\Q$ be a principally polarized abelian variety of dimension $g$ such  that $\Im \rho_A$ is open in $\GSp_{2g}(\hat{\Z})$.
We will investigate $\nu(a_{1, p})$  via the method of moments, 
with the goal of proving:
\begin{proposition}\label{prop-kth-moment}
Assume GRH. Then
\begin{equation}\label{kth-moment}
\frac{1}{\pi(x)}
\ds\sum_{
p \leq x
\atop{
p \nmid N_A
\atop{a_{1, p} \neq 0}
}
}
\left(
\nu(a_{1, p}) - \log \log x
\right)^k
=
c_k (\log \log x)^{\frac{k}{2}} 
+ \o\left((\log \log x)^{\frac{k}{2}}\right)
\end{equation}
for each integer $k \geq 1$,
where
$$
c_k :=
\left\{
\begin{array}{cc}
\frac{k!}{2^{\frac{k}{2}}  \left(\frac{k}{2}\right)!} & \text{if $k$ even}
\\
0 & \text{if $k$ odd}
\end{array}
\right.
$$
is the $k$-th moment of the standard Gaussian.
\end{proposition}
\noindent
With this, by adapting to our context the proof of the Erd\"{o}s-Kac Theorem due to P.~Billingsley \cite{Bi74} 
(see also \cite{Bi69} and the references therein for an accessible exposition),
Theorem \ref{main-thm2} is proved.

The core ingredient in our proof is the following 
application of  (\ref{ramification}) - (\ref{degree}), Theorem  \ref{cheb2} (under GRH) and Proposition \ref{hensel}:
for any positive integer $m$ and any $x > 0$ (to be thought of as approaching infinity), we have
\begin{equation}\label{cheb-applic}
\pi_{C(m, 0)}(x, \Q(A[m])/\Q)
=
\frac{
|C(m, 0)|
}{
|G(m)|
}
 \pi(x)
+
\O\left(
|C(m, 0)| x^{\frac{1}{2}} \log (m N_A x)
\right).
\end{equation}

Related to this, remark  that 
by the openness assumption of $\Im \rho_A$ in $\GSp_{2g}(\hat{\Z})$
and by (\ref{quotient-first}) from the proof of part (iii) of Lemma \ref{open-torsion-cond},
we have 
\begin{equation}\label{ellsqeq}
\frac{
|C(\ell, 0)|
}{
|G(\ell)|
}
=
\frac{1}{\ell} + \O\left(\frac{1}{\ell^2}\right)
\end{equation}
for all $\ell \nmid m_A$.
 In particular, for any $y > 0$,
\begin{equation}\label{sum-density}
\ds\sum_{\ell \leq y} \frac{|C(\ell, 0)|}{|G(\ell)|} = \log \log y + \O_A(1),
\end{equation}
and, after using \eqref{ellsqeq} and \eqref{degree}, 
\begin{equation}\label{sum-class}
\ds\sum_{\ell \leq y} |C(\ell, 0)|
\ll
\frac{
y^{2g^2 + g + 1}
}{
\log y
}.
\end{equation}

Crucial to the method is also the following simple observation. 
Let $x> 0$ and $0 < \delta < 1$  be fixed and let  $y := x^{\delta}$.  For any integer $m \geq 1$, we have
\begin{equation}\label{obs-nu}
|\nu(m) - \nu_y(m)| \leq \frac{\log m}{ \delta \log x},
\end{equation}
where
$\nu_y(m)$ denotes the number of distinct prime divisors $\ell \leq y$ of $m$.

We now proceed with the proof of (\ref{kth-moment}).
For each prime $\ell$, we define a random variable $R_{\ell}$ to be  
$1$ with probability $\frac{1}{\ell}$ and $0$ with probability $1 - \frac{1}{\ell}$.
Upon taking $y :=  x^{\delta}$ for some fixed $0 < \delta < 1$ and
 $x \rightarrow \infty$, 
 $
 R(y) := \ds\sum_{\ell \leq y} R_{\ell}
 $ 
 becomes normally distributed with mean and variance each equal to  $\log \log x$;
 by the Central Limit Theorem, for any integer $k \geq 1$
 we have
\begin{equation}\label{model}
\E\left(  (R(y) - \log \log x)^k \right)
=
c_k (\log \log x)^{\frac{k}{2}}
+ \o\left((\log \log x)^{\frac{k}{2}}\right).
\end{equation}
Note that by (\ref{ellsqeq}),  $R_{\ell}$ models the event that $\ell | a_{1, p}$ for some $p$.
Our strategy then is to  prove (\ref{kth-moment}) by comparing 
$\E\left(  (R(y) - \log \log x)^k \right)$
and
$
\frac{1}{\pi(x)}
\ds\sum_{
p \leq x
\atop{
p \nmid N_A
\atop{
a_{1, p} \neq 0
}
}
}
(\nu_y(a_{1, p}) - \log \log x)^k
$
 for each $k \geq 1$.

We fix  $x > 0$ and $k \geq 1$, choose a parameter $\delta  = \delta(g, k)$ such that 
\begin{equation}\label{delta-choice}
0 < \delta < \frac{1}{2 k \left(2g^2 + g + 1\right)},
\end{equation}
and define $y := x^{\delta}$.  In what follows, our $\O$-estimates will reflect the growth of various functions as $x \rightarrow \infty$.

For each $\ell$ and each $p \nmid N_A$, we define
$$
\delta_{\ell}(p) :=
\left\{
\begin{array}{ll}
1 & \text{if $\ell | a_{1, p}$,}
\\
0 & \text{else.}
\end{array}
\right.
$$
Then, for each integer $1 \leq j \leq k$,
upon applying (\ref{cheb-applic}) - (\ref{sum-class}) and ({\ref{model}),  we obtain
\begin{eqnarray*}
&&
\ds\sum_{
p \leq x
\atop{
p \nmid N_A
\atop{
a_{1, p} \neq 0
}
}
}
\nu_y(a_{1, p})^j
\\
&=&
\ds\sum_{
p \leq x
\atop{
p \nmid N_A
\atop{
a_{1, p} \neq 0
}
}
}
\ds\sum_{
\ell_1, \ldots, \ell_j \leq y
}
\delta_{\ell_1}(p) \ldots \delta_{\ell_j}(p)
\\
&=&
\ds\sum_{
\ell_1, \ldots, \ell_j \leq y
}
\#\left\{
p \leq x: p \nmid N_A, a_{1, p} \neq 0, \lcm\{\ell_1, \ldots, \ell_j\} | a_{1, p}
\right\}
\\
&=&
\ds\sum_{
\ell_1, \ldots, \ell_j \leq y
}
\#\left\{
p \leq x: p \nmid N_A,  \lcm\{\ell_1, \ldots, \ell_j\} | a_{1, p}
\right\}
+
\O(\pi(y)^j  \; \pi_A(x, 0))
\\
&=&
\ds\sum_{
\ell_1, \ldots, \ell_j \leq y
}
\#\left\{
p \leq x: p \nmid  \lcm\{\ell_1, \ldots, \ell_j\} N_A,  \lcm\{\ell_1, \ldots, \ell_j\} | a_{1, p}
\right\}
\\
&&
+
\ds\sum_{
\ell_1, \ldots, \ell_j \leq y
}
\#\left\{
p \leq x: p \nmid N_A,  p | \lcm\{\ell_1, \ldots, \ell_j\} | a_{1, p}
\right\}
+
\O(\pi(y)^j  \; \pi_A(x, 0))
\\
&=&
\ds\sum_{
\ell_1, \ldots, \ell_j \leq y
}
\pi_{C(\lcm\{\ell_1, \ldots, \ell_j\} , 0)} (x, \Q(A[\lcm\{\ell_1, \ldots, \ell_j\} ])/\Q)
+
\O\left(j \pi(y)^j \right)
+
\O\left(\pi_A(x, 0) \; \pi(y)^j \right)
\\
&=&
\E(R(y)^j) \;  \pi(x)
+
\O_j\left(\pi(x) \; (\log \log y)^{j-1}
\right)
+
\O_{A, j}\left(
\frac{y^{j (2 g^2 + g + 1)} }{(\log y)^j}
x^{\frac{1}{2}} \log x
\right)
+
\O_j\left(\pi(y)^j \pi_A(x, 0)\right).
\end{eqnarray*}
By  the binomial theorem and the above, we deduce
\begin{eqnarray*}
&&
\ds\sum_{
p \leq x
\atop{
p \nmid N_A
\atop{
a_{1, p} \neq 0
}
}
}
(\nu_y(a_{1, p}) - \log \log x)^k
\\
&=&
\ds\sum_{0 \leq j \leq k}
{k \choose j} (- \log \log x)^{k - j}
\ds\sum_{
p \leq x
\atop{
p \nmid N_A
\atop{
a_{1, p} \neq 0
}
}
}
\nu_y(a_{1, p})^j
\\
&=&
\ds\sum_{0 \leq j \leq k}
{k \choose j}
(- \log \log x)^{k-j} \;
\E(R(y)^j)
\;
\pi(x)
\\
&+&
\O_{A, k}\left(
y^{k (2g^2 + g + 1)} x^{\frac{1}{2}} (\log x) \; (\log \log x)^k
\right)
+
\O_k\left(
\pi(y)^k \;  \pi_A(x, 0) \; (\log \log x)^k
\right).
\end{eqnarray*}
Recalling the choice of $\delta$ given in (\ref{delta-choice}) and using part (i2) of Theorem \ref{main-thm1}, 
we see that the two $\O$-terms above become 
$\O_{\varepsilon, A, k}\left(
x^{1 - \varepsilon} (\log x) (\log \log x)^k
\right)$,
which is
$\o\left(\pi(x) (\log \log x)^{\frac{k}{2}}\right)$. 
Then, upon applying the binomial theorem once again in order to rewrite the first term, we deduce
\begin{eqnarray}\label{kth-moment-y}
\frac{1}{\pi(x)}
\ds\sum_{
p \leq x
\atop{
p \nmid N_A
\atop{
a_{1, p} \neq 0
}
}
}
(\nu_y(a_{1, p}) - \log \log x)^k
\sim
\E( (R(y) - \log \log x)^k).
\end{eqnarray}

Finally, recalling (\ref{bound-a1p}) and (\ref{obs-nu}) and applying (\ref{model}) and (\ref{kth-moment-y}) several times,  we deduce
\begin{eqnarray*}
&&
\frac{1}{\pi(x)}
\ds\sum_{
p \leq x
\atop{
p \nmid N_A
\atop{a_{1, p} \neq 0}
}
}
\left(
\nu(a_{1, p}) - \log \log x
\right)^k
\\
&=&
\frac{1}{\pi(x)}
\ds\sum_{
p \leq x
\atop{
p \nmid N_A
\atop{a_{1, p} \neq 0}
}
}
\left(
\nu_y(a_{1, p}) - \log \log x
+
\O\left(\frac{\log |a_{1, p}|}{\delta \log x} \right)
\right)^k
\\
&=&
\frac{1}{\pi(x)}
\ds\sum_{
p \leq x
\atop{
p \nmid N_A
\atop{a_{1, p} \neq 0}
}
}
\left(
\nu_y(a_{1, p}) - \log \log x
\right)^k
+
\O_{k, g, \delta}\left(
\ds\sum_{0 \leq j \leq k-1}
\frac{1}{\pi(x)}
\ds\sum_{
p \leq x
\atop{
p \nmid N_A
\atop{a_{1, p} \neq 0}
}
}
\left|\nu_y(a_{1, p}) - \log \log x
\right|^j
\right)
\\
&=&
\frac{1}{\pi(x)}
\ds\sum_{
p \leq x
\atop{
p \nmid N_A
\atop{a_{1, p} \neq 0}
}
}
\left(
\nu_y(a_{1, p}) - \log \log x
\right)^k
+
\O_k\left(
(\log \log x)^{\frac{k-1}{2}}
\right)
\\
&=&
c_k (\log \log x)^{\frac{k}{2}}
+
\o\left(
(\log \log x)^{\frac{k}{2}}
\right).
\end{eqnarray*}
This completes the proof of Theorem \ref{main-thm2}.

\begin{remark}
{\emph{
The first and second moments of $\nu(a_{1, p})$ may be estimated directly, without any comparison with the model defined by $R_{\ell}$.
The strategy originates in P.~Tur\'{a}n's proof of the Hardy-Ramanujan Theorem, \cite{Tu},  and is summarized below.
}}
\end{remark}

We choose $0 < \delta < \frac{1}{8 g^2 + 4 g + 1}$ and let $y = x^{\delta}$. Then, proceeding as in the proof of Theorem \ref{main-thm2},
but without the model $R_{\ell}$, we obtain

\begin{eqnarray}\label{loglogx}
&&
\ds\sum_{
p \leq x
\atop{
p \nmid N_A
\atop{
a_{1, p} \neq 0
}
}
}
\left(
\nu(a_{1, p}) - \log \log x
\right)^2
\nonumber
\\
&=&
\ds\sum_{
p \leq x
\atop{
p \nmid N_A
\atop{
a_{1, p} \neq 0
}
}
}
\nu(a_{1, p})^2
-
2 (\log \log x)
\ds\sum_{
p \leq x
\atop{
p \nmid N_A
\atop{
a_{1, p} \neq 0
}
}
}
\nu(a_{1, p})
+
(\log \log x)^2 \#\{p \leq x: p \nmid N_A, a_{1, p} \neq 0 \}
\nonumber
\\
&=&
\ds\sum_{
p \leq x
\atop{
p \nmid N_A
\atop{
a_{1, p} \neq 0
}
}
}
\left(
\nu_y(a_{1, p}) + \O_{A}(1)
\right)^2
\nonumber
\\
&&
-
2 (\log \log x)
\ds\sum_{
p \leq x
\atop{
p \nmid N_A
\atop{
a_{1, p} \neq 0
}
}
}
\left(
\nu_y(a_{1, p}) + \O_{A}(1)
\right)
+
\pi(x) (\log \log x)^2
+
\O(\pi_A(x, 0) \; (\log \log x)^2)
\nonumber
\\
&=&
\ds\sum_{
p \leq x
\atop{
p \nmid N_A
\atop{
a_{1, p} \neq 0
}
}
}
\nu_y(a_{1, p})^2
-
2 (\log \log x) \ds\sum_{
p \leq x
\atop{
p \nmid N_A
\atop{
a_{1, p} \neq 0
}
}
}
\nu_y(a_{1, p})
+
\pi(x) (\log \log x)^2
\nonumber
\\
&&
+
\O\left(
\ds\sum_{
p \leq x
\atop{
p \nmid N_A
\atop{
a_{1, p} \neq 0
}
}
}
\nu_y(a_{1, p})
\right)
+
\O(\pi(x) \log \log x)
+
\O(\pi_A(x, 0) \; (\log \log x)^2)
\nonumber
\\
&=&
\ds\sum_{
\ell_1, \ell_2 \leq y
\atop{
\ell_1 \neq \ell_2
}
}
\frac{|C(\ell_1 \ell_2, 0)|}{|G(\ell_1 \ell_2)|}
\pi(x)
+
\O_A\left(
\ds\sum_{\ell_1, \ell_2 \leq y} |C(\ell_1 \ell_2, 0)|
x^{\frac{1}{2}}
\log x
\right)
- 2 (\log \log x) \ds\sum_{\ell \leq y} \frac{|C(\ell, 0)|}{|G(\ell)|} \pi(x)
\nonumber
\\
&&
+
\O_A\left(
\ds\sum_{\ell \leq y} |C(\ell, 0)| x^{\frac{1}{2}} (\log x) (\log \log x)
\right)
+ 
\pi(x) (\log \log x)^2
+
\O_A\left(
\ds\sum_{\ell \leq y} \frac{|C(\ell, 0)|}{|G(\ell)|} \pi(x)
\right)
\nonumber
\\
&&
+
\O_A\left(
\ds\sum_{\ell \leq y} |C(\ell, 0)| x^{\frac{1}{2}} \log x
\right)
+
\O(\pi(x) \log \log x)
+
\O(\pi_A(x, 0) \; (\log \log x)^2)
\nonumber
\\
&=&
\pi(x)(\log \log x)^2
+
\O_{A}\left(
\frac{  x^{2 \delta (2 g^2 + g + 1)}   }{\log x}
x^{\frac{1}{2}}
\right)
-  \; 
2 \pi(x) (\log \log x)^2
+
\O\left(
x^{\delta(2 g^2 + g + 1)} \;  x^{\frac{1}{2}} \log \log x
\right)
\nonumber
\\
&&
+ \; 
\pi(x) (\log \log x)^2
+
\O_A\left(
\pi(x) \log \log x
\right)
+
\O_A\left(
x^{\delta(2 g^2 + g + 1)} \;  x^{\frac{1}{2}}  \log x
\right)
+
\O_A\left(
\pi_A(x, 0) (\log \log x)^2
\right)
\nonumber
\\
&=&
\O_{A}\left(
\pi(x) \log \log x
\right).
\end{eqnarray}

Note that the cancellation of the $\pi(x) (\log \log x)^2$ terms is essential and that the choice of $\delta$ ensures that the largest 
emerging $\O$-term  depending  on $y$, namely 
$\O_{A}\left(
\frac{  x^{2 \delta (2 g^2 + g + 1)}   }{\log x}
x^{\frac{1}{2}}
\right)$,
is 
sufficiently small;
precisely, it is $\ll_A \pi(x) \ll_A \pi(x) \log \log x$.

\begin{remark}
{\emph{
That $\nu(a_{1, p})$ has normal order $\log \log p$ can be deduced easily from the second moment estimate (\ref{loglogx}). In particular, this is an immediate consequence of the following variation of (\ref{loglogx}):
$$
\ds\sum_{
p \leq x
\atop{
p \nmid N_A
\atop{
a_{1, p} \neq 0
}
}
}
\left(
\nu(a_{1, p}) - \log \log p
\right)^2
\ll_A
\pi(x) \log \log x.
$$
In turn, this is obtained by remarking that
\begin{eqnarray*}
\ds\sum_{
p \leq x
\atop{
p \nmid N_A
\atop{
a_{1, p} \neq 0
}
}
}
\left(
\nu(a_{1, p}) - \log \log p
\right)^2
&\ll&
\ds\sum_{
p \leq x
\atop{
p \nmid N_A
\atop{
a_{1, p} \neq 0
}
}
}
\left(
\nu(a_{1, p}) - \log \log x
\right)^2
+
\ds\sum_{
p \leq x
\atop{
p \nmid N_A
\atop{
a_{1, p} \neq 0
}
}
}
\left(
\log \frac{\log x}{\log p}
\right)^2,
\end{eqnarray*}
 using (\ref{loglogx}) for the first sum
 and
splitting the last sum over $p$ into a sum over
$p \leq \sqrt{x}$ and one over $\sqrt{x} < p \leq x$, followed by elementary estimates.
}}
\end{remark}

\begin{remark}
{\emph{
The normal order of $\nu(a_{1, p})$ may also be obtained via the ubiquitous large sieve; see   \cite[Proposition 2.15]{Ko} for generalities related to such works.
Moreover, the $k$-th moments
(\ref{kth-moment})
may  be estimated more precisely via sieve methods by applying the general result   \cite[Proposition 3]{GrSo}. 
}}
\end{remark}

\section{Heuristic reasoning for Conjecture \ref{LT-dimg}}

We devote this section to arguing heuristically towards Conjecture \ref{LT-dimg}.
Our main setting will be that of  a principally polarized abelian variety $A/\Q$ of dimension $g$
for which $\Im \rho_A$ is open in $\GSp_{2g}(\hat{\Z})$ and which satisfies
the Equidistribution Assumption.
In particular, the function $\Phi$ introduced in Section \ref{introsect} is bounded, continuous, and nonzero on $(-1, 1)$; 
this was proved in more than one way in email communication between 
N.~Katz \cite{Ka15} and J-P.~Serre; in Appendix \ref{SerreApp2} we include 
a letter from Serre to Katz that contains such a proof.

\begin{definition}\label{def-fp}
For each integer $m \geq 1$ and prime $p$, define $c_{p,m} \in (0, \infty)$ by
$$
c_{p, m} = 
\frac{|G(m)|}{m 
\ds\sum_{\tau \in \Z \atop{|\tau| < 2 g \sqrt{p}}} 
\Phi\left(\frac{\tau}{2 g \sqrt{p}}\right)  |C(m, \tau)|}
$$
and define the function
$$
f_p^{(m)} :  \Z \longrightarrow [0, \infty),
$$

$$
f_p^{(m)}(\tau) :=
\left\{
\begin{array}{cc}
        \Phi \left(\frac{\tau}{2 g \sqrt{p}}\right) \cdot \frac{m |C(m, \tau)|}{|G(m)|} \cdot c_{p,m} \; \; 
      &  \text{if $|\tau| < 2 g \sqrt{p}$,}
\\
0  & \; \text{else.}
\end{array}
\right.
$$
\end{definition}

Note that 
$$
\ds\sum_{\tau \in \Z} f_p^{(m)}(\tau) = 1.
$$

\begin{lemma}\label{ST-a1p}
For all integers $m \geq 1$ and $\tau_0 \in \Z$  we have
$$
\ds\lim_{p \rightarrow \infty}
\frac{m}{2 g \sqrt{p}} 
\ds\sum_{
\tau \in \Z
\atop{
|\tau| < 2 g \sqrt{p}
\atop{
\tau \equiv \tau_0 (\mod m)
}
}
}
\Phi \left(\frac{\tau}{2 g \sqrt{p}}\right)
=
1.
$$
\end{lemma}
\begin{proof}
This follows by viewing the expression inside the limit as a Riemann sum approximation of the integral
$\ds\int_{-1}^{1} \Phi (\tau) \; d \tau = 1$.
For more details, 
 see \cite[pp.  31--32]{LaTr}.
\end{proof}

\begin{lemma}\label{cp-g}
For all integers $m \geq 1$ we have
$$
\ds\lim_{p \rightarrow \infty} {2 g \sqrt{p} \;  c_{p, m}} = 1.
$$
\end{lemma}
\begin{proof}
By the definition of $c_{p, m}$ and Lemma \ref{ST-a1p},
\begin{eqnarray*}
\ds\lim_{p \rightarrow \infty}
\frac{1}{2 g \sqrt{p} \; c_{p, m}}
&=&
\ds\lim_{p \rightarrow \infty}
\frac{1}{2 g \sqrt{p}}
\ds\sum_{\tau_0 =0}^{m-1}
 \ds\sum_{\tau \in \Z
  \atop{|\tau| < 2 g \sqrt{p}
  \atop{\tau \equiv \tau_0 (\mod m)}
  }
  } 
 \Phi \left(\frac{\tau}{2 g \sqrt{p}}\right) \frac{m \; |C(m, \tau)|}{|G(m)|}
 \\
 &=&
\ds\lim_{p \rightarrow \infty}
\ds\sum_{\tau_0 =0}^{m-1}
 \frac{|C(m, \tau_0)|}{|G(m)|}
 \left(
 \frac{m}{2 g \sqrt{p}}
 \ds\sum_{\tau \in \Z
  \atop{|\tau| < 2 g \sqrt{p}
  \atop{\tau \equiv \tau_0 (\mod m)}
  }
  } 
 \Phi \left(\frac{\tau}{2 g \sqrt{p}}\right) 
 \right)
 \\
 &=&
\ds\sum_{\tau_0 =0}^{m-1}
 \frac{|C(m, \tau_0)|}{|G(m)|}
 \\
 &=&
 1.
\end{eqnarray*}
\end{proof}

 Now let us fix $t \in \Z$ and 
assume that
  $\ds\lim_{m \; \widetilde{\rightarrow} \; \infty} f_p^{(m)}(t)$
models the likelihood of the event  $a_{1, p} = t$, as guided by the Chebotarev law for all $m$-division fields and by  the behaviour of $\frac{a_{1,p}}{2g\sqrt{p}}$ in the interval $(-1,1)$.
Then, recalling part (iv) of Lemma \ref{open-torsion-cond} and  Lemma \ref{cp-g},
we  reason \emph{heuristically} as follows:
\begin{align*}
\#\{p &\leq x : p \nmid N_A, a_{1, p} = t\} \\
&\approx
\ds\lim_{m \; \widetilde{\rightarrow} \; \infty}
\ds\sum_{p \leq x} f_p^{(m)}(t)
\\
&=
\ds\lim_{m \; \widetilde{\rightarrow} \; \infty}
\ds\sum_{p \leq x}
\Phi \left(\frac{t}{2 g \sqrt{p}}\right)
\cdot
\frac{m |C(m, t)|}{|G(m)|} \cdot c_{p,m}
\\
&\approx 
\left(
\ds\lim_{m \; \widetilde{\rightarrow} \; \infty}
\frac{m |C(m, t)|}{|G(m)|}
\right)
\ds\sum_{p \leq x}
\Phi \left(\frac{t}{2 g \sqrt{p}}\right)
\cdot
\frac{1}{2 g \sqrt{p}}
\\
&=
\left(
\lim_{m \; {\widetilde{\rightarrow}}  \;  \infty} F_t(m)
\right)
\;
\ds\sum_{p \leq x}
\Phi \left(\frac{t}{2 g \sqrt{p}}\right)
\cdot
\frac{1}{2 g \sqrt{p}}.
\end{align*}
Here, the symbol $\approx$ means equality deduced purely heuristically. 
The last line is simply notation, as introduced in Section 2.2.

To understand the growth of  the last sum, we  use the properties of the function $\Phi$. For any $\varepsilon > 0$,
by the continuity of $\Phi$ at $0$, there exists a $\delta > 0$ such that
\begin{equation}\label{continuity}
\left|\frac{t}{2 g \sqrt{p}}\right|  < \delta \; \Rightarrow \; 
\left|
\Phi \left(\frac{t}{2 g \sqrt{p}}\right)
-
\Phi (0)
\right|
< \varepsilon.
\end{equation}
We thus split the sum over $p \leq x$ according to the above $\delta$-interval. By the boundedness of 
$\Phi$, we obtain
$$
\left|
\ds\sum_{p < \frac{t^2}{4 g^2 \delta^2} }
\left(
\Phi\left(\frac{t}{2 g \sqrt{p}}\right)
-
\Phi(0)
\right)
\frac{1}{2  \sqrt{p}}
\right|
\ll_{t, \varepsilon, g} 1.
$$
By (\ref{continuity}) and by noting that $\ds\sum_{p \leq x} \frac{1}{2 \sqrt{p}} \sim \frac{\sqrt{x}}{\log x}$,
we obtain
$$
\left|
\ds\sum_{\frac{t^2}{4 g^2 \delta^2}  < p \leq x}
\left(
\Phi\left(\frac{t}{2 g \sqrt{p}}\right)
-
\Phi(0)
\right)
\frac{1}{2 \sqrt{p}}
\right|
\ll
\frac{\varepsilon  \sqrt{x} }{\log x}.
$$
Taking $\varepsilon \rightarrow 0$ and returning to our heuristics, 
we are led to the possible prediction that
\begin{equation}\label{precise-conj}
\#\{p \leq x: p \nmid N_A, a_{1, p} = t\}
\sim
\frac{\Phi(0)}{g}
\cdot
\lim_{m \; {\widetilde{\rightarrow}}  \;  \infty} F_t(m)
\cdot
\frac{\sqrt{x}}{\log x}.
\end{equation}

When $t \neq 0$, we proved 
in parts (iii) and  (iv) of  Lemma \ref{open-torsion-cond} that
the limit over $m \; \widetilde{\rightarrow} \; \infty$ exists and equals an 
 infinite product;  in this case,  we conjecture
that

\begin{multline}\label{precise-conj-bis}
\#\{p \leq x: p \nmid N_A, a_{1, p} = t\}
\sim \\
\frac{\Phi(0)}{g}
\cdot
\frac{
m_{A, t}
|C(m_{A, t}, t)|
}{|G(m_{A, t})|}
\cdot
\ds\prod_{\ell \nmid m_A}
\frac{
\ell^{v_\ell(t)+1}
\;
|
\{M \in \GSp_{2g}(\Z/\ell^{v_\ell(t)+1} \Z) : \tr M \equiv t (\mod \ell^{v_\ell(t)+1})\}
|
}{ |\GSp_{2g}(\Z/\ell^{v_\ell(t)+1} \Z)| }
\cdot
\frac{\sqrt{x}}{\log x},
\end{multline}
where we recall
$$
m_{A, t} =
m_A
\ds\prod_{\ell | m_A} \ell^{v_{\ell}(t)}.
$$
When $t = 0$ and $g = 1$, 
the limit over $m \; \widetilde{\rightarrow} \; \infty$ exists and equals an 
infinite product by \cite[Lemma 2 p.~34]{LaTr}; see Remark  \ref{g=1} below.
When $t = 0$ and $g \geq 2$, 
we are currently unable to make a similar statement and relegate such a study to future work.

\bigskip

We conclude this section with several remarks about the above conjecture.

\begin{remark}\label{g=1}
{\emph{
Assume $g = 1$ and $\End_{\overline{\Q}}(E) \simeq \Z$.
Then  $\Im \rho_A$ is open in $\GSp_{2g}(\hat{\Z})$
(see 
\cite{Se72}) 
and  the Equidistribution Assumption holds 
(see \cite{BLGeHaTa}, 
\cite{ClHaTa},  and \cite{Cl}). 
In this case, $\Phi(x) = \frac{2}{ \pi} \sqrt{1 - x^2}$
and
(\ref{precise-conj}) 
 coincides with the formulation in (\ref{LT-dim1}) of the Lang-Trotter Conjecture on Frobenius traces of  \cite{LaTr}.
 Combining this with the formula
 $$
|\GL_2(\Z/\ell \Z)| = \ell (\ell-1)\left(\ell^2 - 1\right)
 $$
 and with
 \cite[Lemma 2 p. 34]{LaTr},
 we obtain an equivalent reformulation of (\ref{precise-conj}):
 \\
 \begin{equation}\label{precise-conj-1}
                 \pi_A(x,t) \sim
                \frac{2}{\pi}
\cdot
\frac{m_A |C(m_A, t)|}{|G(m_A)|}
\cdot
\ds\prod_{\ell \nmid m_A \atop{ \ell \mid t} }
\frac{\ell^2}{ \ell^2-1}
\cdot
\ds\prod_{\ell \nmid tm_A}
\frac{\ell \left(\ell^2 - \ell - 1\right)}{(\ell + 1) (\ell - 1)^2}
\cdot 
\frac{\sqrt{x}}{\log x}.
 \end{equation}
 \\
 }}
\end{remark}

 \begin{remark}\label{Phi-g2} 
 {\emph{
Assume $g=2$ and $\End_{\overline{\Q}}(A) \simeq \Z$.
Then 
$\Im \rho_A$ is open in $\GSp_{2g}(\hat{\Z})$
(\cite{Se86}, \cite{Se86bis})
and the Sato-Tate group of $A$ is 
$\USp(4)$ (\cite[Theorem~4.3]{FiKeRoSu}),
while 
the Equidistribution Assumption is an open question. 
The function  $\Phi(\cdot)$  may  be calculated explicitly using the Weyl integration formula as in \cite{KeSu}.
In particular, this calculation leads to the value
$$
\Phi(0) = 
\frac{256}{15 \pi^2}.
$$
We explain the calculation of $\Phi(0)$ here briefly.
Let
 $
L_p(A, T) := T^{4} P_{A, p} \left(\frac{1}{T}\right)
$
be the $p$-Euler factor in the $L$-function of $A$
and let  
$
\bar{L}_p(A, T) 
=
L_p\left(A, \frac{T}{\sqrt{p}}\right)
$
be its normalization.
Let
$$
S :=
\left\{
(x_1, x_2) \in \R^2 :
x_2 \geq 2 x_1 - 2, 
x_2 \geq - 2 x_1 - 2, 
x_2 \leq \frac{x_1^2}{4} + 2
\right\}
$$
and
let $R(x_1)$ be the defining interval of $x_2$ imposed by the constraints of $S$.
Recalling
that  the  Sato-Tate group associated to $A$ is 
$\USp(4)$,
  the conjectured joint density function of the normalized coefficients
$\bar{a}_{1, p}$ and $\bar{a}_{2, p}$
 is 
$$
\frac{1}{4 \pi^2} \sqrt{\max\{\rho(\bar{a}_{1, p}, \bar{a}_{2, p}), 0\}},
$$
where
$$
\rho(x_1, x_2)
:=
\left(
x_1^2 - 4 x_2 + 8
\right)
\left(
x_2 - 2 x_1 + 2
\right)
\left(
x_2 + 2 x_1 + 2
\right),
$$
with support in the region  $S$
where $\rho$ is non-negative.
Consequently, for any interval $I \subseteq [-4, 4]$, the set
$$
\left\{
p: \bar{a}_{1, p} \in I
\right\}
$$
is expected to have  natural density
$$
\ds\int_{I}
\ds\int_{R(x_1)}
\frac{1}{4 \pi^2}
\sqrt{\max\{\rho(x_1, x_2), 0\}} \; d x_2 \; d x_1.
$$
(For details, see the original source, specifically  \cite[p.~21 and p.~40]{FiKeRoSu}.) 
Let
$$
\Psi(x) = \frac{1}{4 \pi^2} 
\ds\int_{R(x)}
\sqrt{
\max\{\rho(x, x_2), 0\}
}
\, d x_2.
$$
In particular,
$R(0) = [-2, 2]$ and
$$
\Psi(0) = 
\frac{1}{4 \pi^2} \ds\int_{-2}^2 \sqrt{(8 - 4 x_2) (x_2 + 2)^2} \; d x_2
=
\frac{64}{15 \pi^2}.
$$
In our notation $\Phi(x)=\Psi(4x)\cdot 4$, since one can rescale the variable and account for the fact that both functions are assumed to have integral 1. Therefore, $\Phi(0)=\Psi(0) \cdot 4=\frac{256}{15 \pi^2}$.
}}
\end{remark}

\begin{remark}\label{example}
{\emph{
For $g=2$, $t = \pm 1$, and $\End_{\overline{\Q}}(A) \simeq \Z$,
we have an equivalent reformulation of (\ref{precise-conj-bis}): 
 \[
                \pi_A(x,t) \sim
                \frac{128}{15\pi^2}
\cdot
\frac{m_A |C(m_A, t)|}{|G(m_A)|}
\cdot
\ds\prod_{\ell \nmid m_A}
\frac{
        \ell(\ell^6 - \ell^5 - \ell^4 + \ell + 1)
}{ (\ell-1)(\ell^2-1)(\ell^4-1)  }
\cdot 
\frac{\sqrt{x}}{\log x}.
        \]
        This is obtained by combining \eqref{eqn:g2clt} with the value of $\Phi(0)$ from the previous remark 
        and with the formula $$|\GSp_4(\Z/\ell\Z)| = \ell^4(\ell-1)(\ell^2-1)(\ell^4-1).$$
}}
\end{remark}

\begin{remark} 
{\emph{
                For higher $g$, the function $\Phi$ is shown to have a certain general form in Appendix \ref{SerreApp2}.  It may again be calculated explicitly 
using for example \cite[Theorem 7.8.B]{We} and 
\cite[5.0.4]{KS} 
(see also the upcoming \cite{BuFiKe}).
}}
\end{remark}

\begin{remark}
{\emph{
 For $g=1$, a more refined version of (\ref{LT-dim1}) was  proposed in \cite{BaJo};
for higher $g$, similar refinements  are relegated to future work.
}}
\end{remark}

\begin{remark}\label{generalST}
{\emph{
Variations of our Conjecture \ref{LT-dimg} may be 
formulated for non-generic classes of abelian varieties such as the case of a CM elliptic curve $E/\Q$
(which was already considered in \cite{LaTr});
in such cases, both the assumption on the image of $\rho_A$ and
the Equidistribution Assumption
must be modified 
appropriately. 
We relegate such endeavours to future work.
}}
 \end{remark}

\section{Computations}

The Lang-Trotter Conjecture as formulated in (\ref{precise-conj-1}) has been supported by numerical evidence (see \cite{LaTr}, \cite{CaHuJaJoScSm}, and \cite{CoFiInYi}). 
Among the main ensuing difficulties are the computations of the integer $m_A$ and 
of the quotient $\frac{m_A |C(m_A, t)|}{|G(m_A)|}$. These
may be resolved for $g=1$ by working with a {Serre curve}, i.e., an elliptic curve for which 
$\left|\GL_2\left(\hat{\Z}\right) : \Im \rho_A\right| = 2$. For such a curve,
the integer $m_A$ is the least common multiple of $2$ and the discriminant of $\Q\left(\sqrt{\Delta_A}\right)$,
where $\Delta_A$ is the discriminant  of any Weierstrass equation of $A$; see \cite[Section 4, p.~1558]{Jo}.
As proved in \cite{Jo} and later  in \cite{CoGrJo}, in more than one sense almost all elliptic curves are Serre curves.  Examples of such curves, as exhibited by Serre in \cite[pp.~310--311]{Se72} and by H. Daniels in 
\cite[p.~227]{Da}, have  been used for numerical computations in \cite{LaTr} and \cite{CoFiInYi}.

For higher $g$, the investigation of $m_A$ from a computational perspective is a solid problem in itself that remains to be tackled.
In this section, while we do not provide numerical evidence for Conjecture \ref{LT-dimg}, we do provide some computational data that complements our main theoretical results.

\subsection{Values of $\pi_A(x,t)$}

Figures \ref{fig:piA} and \ref{fig:piA1} show the values of
$\pi_{A}(x, t)$
 graphed versus $\sqrt{x}/\log x$ for $t \in \{0, 1\}$ and $A \in \{J_1, J_2, J_3\}$, where $J_1, J_2, J_3$   are the Jacobians of the hyperelliptic curves listed in Table \ref{table}. Prediction (\ref{precise-conj}) would  
 imply that these graphs approximate a straight line, whose slope is determined by the constant in front of $\sqrt{x}/\log x$; the graphs are indeed consistent with this implication.

\begin{table}[h]
\begin{center}
\begin{tabular}{c|c|c|c}
        Jacobian &hyperelliptic curve & genus & comments \\
\hline
$J_1$ & $y^2 = x^5 - x + 1$ & 2 & good reduction outside $\{ 2, 19, 151 \}$, $\End(J_1) \cong \Z$  \cite[p.509]{D} \\
$J_2$ & $y^2 = 4x^7 - 12x - 35$ & 3 & everywhere semistable, $\End(J_2) \cong \Z$  \cite[p.2]{Zar} \\

$J_3$ & $y^2 = 4x^9 - 8x - 39$ & 4 & everywhere semistable, $\End(J_3) \cong \Z$  \cite[p.2]{Zar} \\
\hline
\end{tabular}
\end{center}
\caption{Jacobians of hyperelliptic curves used in computations.}
\label{table}
\end{table}
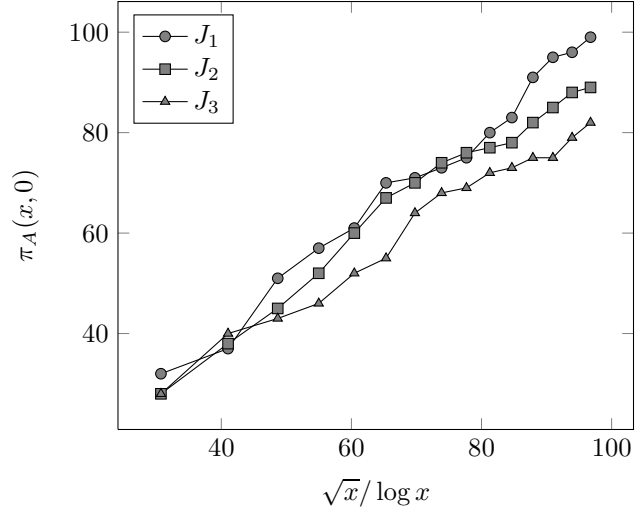
\begin{figure}[h]
\begin{tikzpicture}
\pgfplotscreateplotcyclelist{my black white}{%
solid, every mark/.append style={solid, fill=gray}, mark=*\\%
solid, every mark/.append style={solid, fill=gray}, mark=square*\\%
solid, every mark/.append style={solid, fill=gray}, mark=triangle*\\%
}
\begin{axis}[
xticklabel style=
{anchor=near xticklabel},
title={\;},
xlabel={$\sqrt{x}/\log x$},
y tick label style={/pgf/number format/1000 sep=},
extra y tick style={grid=major, tick label style={xshift=-1cm}},
ylabel={$\pi_A(x,0)$}, legend pos=north west,
cycle list name=my black white
]
\addplot table[x=date,y=value] {curve1.0.dat};
\addplot table[x=date,y=value] {curve4.0.dat};
\addplot table[x=date,y=value] {curve6.0.dat};
\addlegendentry{$J_1$}
\addlegendentry{$J_2$}
\addlegendentry{$J_3$}
\end{axis}
\end{tikzpicture}
\caption{Values of $\pi_A(x,0)$ versus $\sqrt{x}/\log x$ for various Jacobians of hyperelliptic curves.}
\label{fig:piA}
\end{figure}

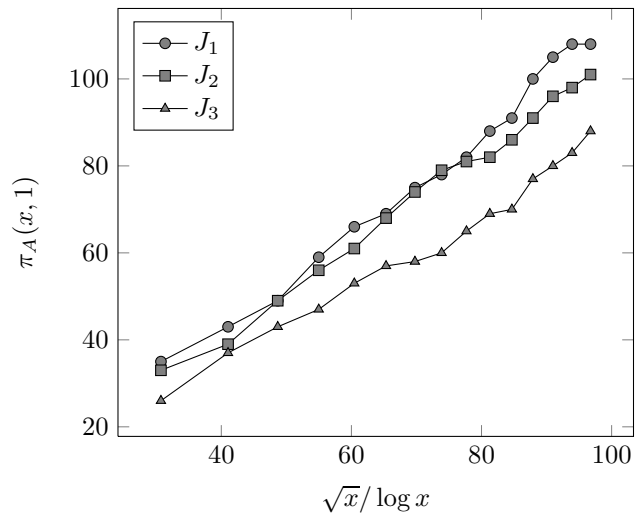
\begin{figure}[h]
\begin{tikzpicture}
\pgfplotscreateplotcyclelist{my black white}{%
solid, every mark/.append style={solid, fill=gray}, mark=*\\%
solid, every mark/.append style={solid, fill=gray}, mark=square*\\%
solid, every mark/.append style={solid, fill=gray}, mark=triangle*\\%
}
\begin{axis}[
xticklabel style=
{anchor=near xticklabel},
title={\;},
xlabel={$\sqrt{x}/\log x$},
y tick label style={/pgf/number format/1000 sep=},
extra y tick style={grid=major, tick label style={xshift=-1cm}},
ylabel={$\pi_A(x,1)$}, legend pos=north west,
cycle list name=my black white
]
\addplot table[x=date,y=value] {curve1.1.dat};
\addplot table[x=date,y=value] {curve4.1.dat};
\addplot table[x=date,y=value] {curve6.1.dat};
\addlegendentry{$J_1$}
\addlegendentry{$J_2$}
\addlegendentry{$J_3$}
\end{axis}
\end{tikzpicture}
\caption{Values of $\pi_A(x,1)$ versus $\sqrt{x}/\log x$ for various Jacobians of hyperelliptic curves.}
\label{fig:piA1}
\end{figure}

\subsection{Converging products of Lemma \ref{open-torsion-cond}}

In part (iii) of  Lemma \ref{open-torsion-cond}, we showed that the following infinite product converges for all 
integers $t$ and  all integers $g \geq 1$:
\[
        P_{g,t} :=  \ds\prod_{\ell }
\frac{
\ell 
\cdot
|
\{M \in \GSp_{2g}(\Z/\ell \Z) : \tr M \equiv t (\mod \ell)\}
|
}{ |\GSp_{2g}(\Z/\ell \Z)| }.
\]
The numerical value of this product depends on the genus $g$ and the primes dividing the trace $t$ (more precisely, the numerator of each factor depends only on whether or not $\ell \mid t$).  It is possible to compute its value for various $g$ and $t$. For example, when $g=2$ we can use the explicit formulae of Remark \ref{remark-g2tracecounts}.  In that case, for $t  \in \{0, 1\}$, numerical computations show that the products appear to quickly converge to
\[
        P_{2,0} \approx 1.3547\ldots, \quad
        P_{2,1} \approx 0.7988\ldots. \quad
\]

\subsection{The normal order of $\nu(a_{1,p})$} 

Figure \ref{fig:normalorder} shows the average number of 
 $\nu(a_{1,p})$ for $J_1$ of Table \ref{table}, for $p$ in the intervals $[2^{i-1}, 2^i]$,
  $i = 2, \ldots, 21$.  The graphs of $\log \log 2^i$ and $\log \log 2^{i-1}$ are shown for comparison.  Figure \ref{fig:normalorder1} presents histograms of the values of $\nu(a_{1,p})$ for $J_1$ in two intervals:  $[1, 2^{20}]$ and $[2^{20},2^{21}]$.  
 The corresponding histograms for $J_2$ and $J_3$ are very similar.

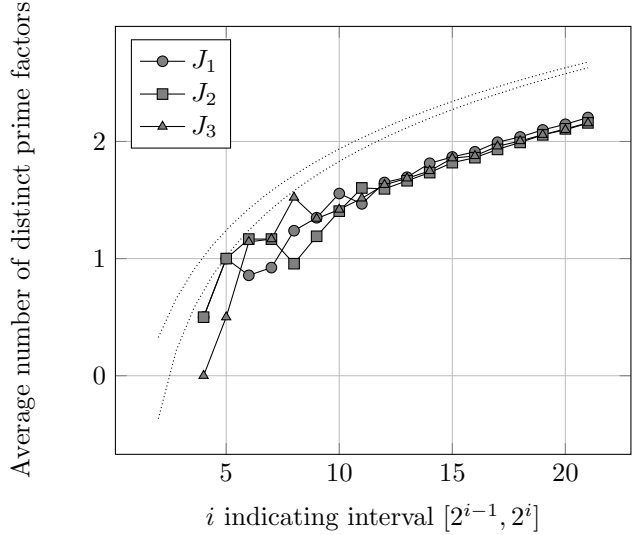
\begin{figure}[h]
\begin{tikzpicture}
\pgfplotscreateplotcyclelist{my black white}{%
solid, every mark/.append style={solid, fill=gray}, mark=*\\%
solid, every mark/.append style={solid, fill=gray}, mark=square*\\%
solid, every mark/.append style={solid, fill=gray}, mark=triangle*\\%
}
\begin{axis}[
xtick={5,10,15,20},
xticklabel style=
{anchor=near xticklabel},
title={\;},
xlabel={$i$ indicating interval $[2^{i-1}, 2^i]$},
y tick label style={/pgf/number format/1000 sep=},
extra y tick style={grid=major, tick label style={xshift=-1cm}},
ylabel={Average number of distinct prime factors}, legend pos=north west,
grid=both,
cycle list name=my black white,
]
\addplot table[x=power,y=value] {normalorder.dat};
\addplot table[x=power,y=value] {normalorder2.dat};
\addplot table[x=power,y=value] {normalorder3.dat};
\addlegendentry{$J_1$}
\addlegendentry{$J_2$}
\addlegendentry{$J_3$}
\addplot +[domain=2:21,no markers,densely dotted] {ln(ln(2^(x-1)))};
\addplot +[domain=2:21,no markers,densely dotted] {ln(ln(2^(x)))};
\end{axis}
\end{tikzpicture}
\caption{Average number of $ \nu(a_{1,p})$ for  $J_1$, $J_2$ and $J_3$, in the intervals $[2^{i-1}, 2^{i}]$, $i=2,\ldots, 21$.  For comparison, the graphs of $\log(\log(2^i))$ and $\log(\log(2^{i-1}))$ are shown in dotted lines.}
\label{fig:normalorder}
\end{figure}

\newpage


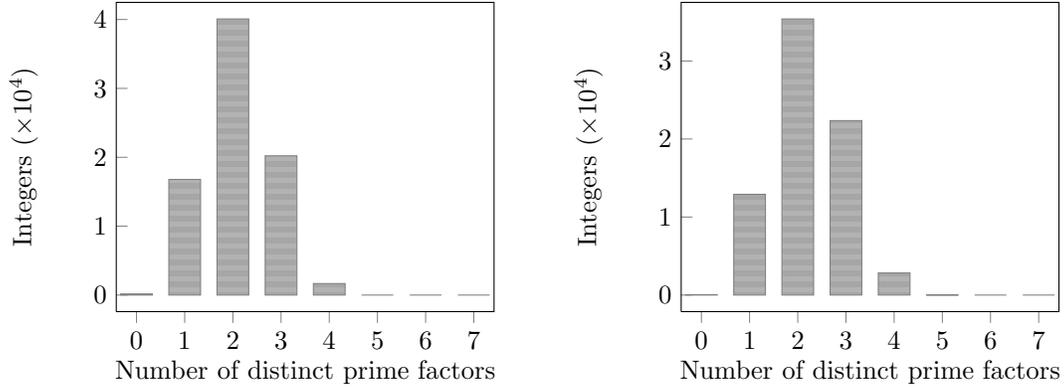
\begin{figure}[h]
 \begin{tikzpicture}
\begin{axis}[
     ybar,
    enlargelimits=0.06,
    width=0.4\textwidth,
    bar width=12pt,
    xtick={0,1,2,3,4,5,6,7},
    xtick pos=left,
    ytick pos=left,
xlabel={Number of distinct prime factors},
ylabel={Integers ($\times 10^4$)},
]
\addplot
[draw=gray,pattern=horizontal lines gray
]
coordinates
{ (0, 0.0166) (1, 1.6787) (2, 4.0083) (3, 2.0214) (4, 0.1673) (5,0) (6,0) (7,0)  };
\end{axis}
\end{tikzpicture}
\quad\quad
\begin{tikzpicture}
\begin{axis}[ybar,
    enlargelimits=0.06,
    width=0.4\textwidth,
    legend style={at={(0.5,-0.2)},
      anchor=north,legend columns=-1},
    bar width=12pt,
    xtick={0,1,2,3,4,5,6,7},
    xtick pos=left,
    ytick pos=left,
xlabel={Number of distinct prime factors},
ylabel={Integers ($\times 10^4$)},
]
\addplot
[draw=gray,pattern=horizontal lines gray
]
coordinates
{ (0, 0.0054) (1, 1.2915) (2, 3.5388) (3, 2.2358) (4, 0.2833) (5, 0.0009) (6, 0) (7, 0) };

\end{axis}
\end{tikzpicture}
\caption[]{Histograms of $\nu(a_{1,p})$ for $J_1$; on the left, the data is for primes $p < 2^{20}$, and on the right, the data is for primes $2^{20} < p < 2^{21}$.  Primes of bad reduction and primes for which the trace is zero are excluded.}
\label{fig:normalorder1}        
\end{figure}


\appendix
\section{Letter by J-P.~Serre on dimension of conjugacy classes in symplectic groups}
\label{SerreApp1}

\def\n{\noindent}
\def\Tr{\mathop{\mathrm{Tr}}}

\def\GSp{\mathop{\mathbf{GSp}}}
\def\Sp{\mathop{\mathbf{Sp}}}
\def\GL{\mathop{\mathbf{GL}}}
\def\SL{\mathop{\mathbf{SL}}}
\def\Z{\mathop{\mathbf{Z}}}
\def\R{\mathop{\mathbf{R}}}

\def\to{\rightarrow}

Paris, May 8, 2015

\bigskip
Dear professor  Cojocaru,

\bigskip

In case you want optimal estimates for the dimensions of conjugacy classes in  $\Sp$  and  $\GSp$,
here is what one can say:

\bigskip

  Let us consider the algebraic groups  $G =  \GSp_{2n}$  or  $\Sp_{2n}$ over a field  $k$  of characteristic $0$ (there are some small changes in char. $p > 0$). Assume $n > 1$, since the case of $\GL_2$ and $\SL_2$ is obvious. If  $g \in G(k)$, let   $d(g)$ be the dimension of the conjugacy class of  $g$, viewed as an algebraic subvariety of~$G$; we have $d(g) = \dim G - \dim Z_G(g)$,
  where  $Z_G(g)$ is the centralizer of  $g$  in  $G$.   
  
  \bigskip

  Theorem - {\it Assume that  $g$ is not of the form  $cu$, where  $c$  is in the center of $G$ and    $u$ is unipotent. Then  $d(g) \geqslant 4n-4$. If moreover  $\Tr(g) = 0$ and $n > 2$, we have
  $d(g) \geqslant 4n-2$.}
  
  \bigskip

  A few remarks before giving the proof:
  
  \medskip
 \n 1) This is a ``geometric'' statement: we may assume that the ground field is algebraically closed.\\
   2) We may assume that  $G = \Sp_{2n}$; the case of $\GSp_{2n}$ follows by writing  $g$ as product of a scalar and an element of  $\Sp_{2n}$; the dimension of the conjugacy class is the same.\\
  3) If  $\Tr(g) = 0$, then the condition `` $g \neq cu$ '' is satisfied, thanks to the fact that
  the characteristic does not divide $2n.$\\
   4) The bounds are optimal. One realizes them by using the obvious embedding  
   $\iota : \SL_2 \to G$, fixing a non-degenerate subspace of codimension 2. If one chooses
   $g = \iota(-1)$, the centralizer of $g$ in  $\Sp_{2n}$ is $\Sp_{2n-2} \times \SL_2$; its dimension is  
   $2(n-1)^2 + n-1 + 3 = 2n^2 -3n +4$; hence the dimension of the conjugacy class of $g$  is  $\dim G - (2n^2-3n+4)  = 2n^2 + n - (2n^2-3n+4) = 4n-4$. If one chooses $g = \iota(x)$, where  $x \in \SL_2$ is such that $\Tr(x) = 2-2n$ (this is always possible, and gives a non central element because $2-2n \neq \pm 2$), one gets 
   an element of trace 0 with centralizer the product of $\Sp_{2n-2} $ by a group of dimension 1; its dimension is  $2n^2 -3n +2$, and the dimension of its conjugacy class is
   $4n-2$.\\
   5) In the $\ell$-adic application, one needs the fact that, if  $g \in \GSp_{2n}(\Z_\ell)$, the dimension (as an $\ell$-adic manifold) of the conjugacy class of  $g$ is the same as its dimension
   in the sense of algebraic geometry.
   
    \n For a given  $t \in \Z$, with $t \neq \pm 2n$, one needs to choose $\ell$ so that no element of trace $t$ of $\GSp_{2n}(\Z_\ell)$ can be of the forbidden shape $cu$ ; as you explain in your paper, this is done by taking $\ell$ such that $v_\ell(t/2n) \neq 0$, which is 
     always possible.

   \bigskip
   
   \n {\it Proof of  $d(g) \geqslant  4n-4$.}
   
   \smallskip
   
     We may assume that  $g$ is semisimple. Indeed, if we decompose  $g$ in Jordan form, as
     $g = su = us$, where  $s$ is semisimple and $u$ is unipotent, the centralizer of  $g$  is contained in the centralizer of  $s$, hence  $d(g) \geqslant d(s)$.

     \smallskip 
     Let us decompose the vector space  $V = k^{2n}$  (with its chosen nondegenerate alternating form) as a direct sum of eigenspaces of $g$, say  $V = \oplus V_{\lambda}$. These spaces have the following properties:
     
   a)    $V_1$ and $V_{-1}$ are non degenerate, hence of even dimension;
   
     b)  If $\lambda \neq 1,-1$, then  $V_\lambda$ is totally isotropic, and in duality with  $V_{\lambda^{-1}}$. 
     
     Put  $n_\lambda = \dim V_\lambda$. The centralizer of $g$ is :
     
     \smallskip 
\quad \quad       $Z_G(g) = \Sp_{n_1} \times \ \Sp_{n_{-1}} \times \  \Pi' \GL_{n_\lambda},$

   \medskip
   
   \n where the symbol $ \Pi'$ means a product on a set $\Lambda$ such that
   $k^\times$ is the disjoint union of $\{1,-1\}, \Lambda$ and $\Lambda^{-1}$.
   
   \smallskip
   {\footnotesize [It might be more efficient to use the orthogonal decomposition of  $V$ given by the eigenspaces of  $g + g^{-1}$.]}
   
   \smallskip
 \n    This implies :
 
   \smallskip
     
  (1)  \quad  $ \dim Z_G(g) = \frac{1}{2}(n_1^2 + n_1 + n_{-1}^2+n_{-1}) + \Sigma' n_\lambda^2$,
  
    \smallskip
      
   \n    where  $\Sigma'$ means a summation over  $\lambda \in \Lambda$.

   \medskip 
 \n  We also have:
   
     \smallskip

(2)  \quad    $2n = n_1 + n_{-1} + 2 \Sigma' n_\lambda.$
   
     \smallskip

 \n   We now need to give an upper bound for the sum (1). Let us simplify the notation by putting 
   $ x = n_1/2, y=n_{-1}/2, z = \Sigma' n_\lambda$. Equation (2) becomes:
   
     \smallskip

   (2') \quad  $n = x + y + z$,
   
     \smallskip

 \n   and  equation (1)  implies (using  $(\Sigma' n_\lambda)^2 \geqslant \Sigma' n_\lambda^2$):
   
     \smallskip

   (3)  \quad $ \dim Z_G(g)  \leqslant  2x^2+x+2y^2+y+z^2$.
   
     \smallskip

\n Consider first the case $z=0$ (i.e. the case where  $g$ as order  2). In that case,
(1) shows that $ \dim Z_G(g)  = 2x^2+x+2y^2+y = 2n^2+n - 4xy$. We have  $xy \neq 0$, otherwise
$g$ would be central; hence  $x$ and $y$ run between  $1$ and $n-1$. In that range the product
$xy$ is minimum when either  $x$ or  $y$ is equal to  1, in which case its value is $n-1$.
Hence $ \dim Z_G(g)  \leqslant 2n^2 + n - 4(n-1)$, i.e. $d(g) \geqslant 4(n-1)$, as wanted.

\smallskip

Suppose $z \geqslant 1$; we have $x,y \geqslant 0$. With the relation
(2'), this shows that the point $(x,y,z) \in \R^3$ belongs to the triangle with vertices
the three points $(0,0,n), (0,n-1,1), (n-1,0,1).$  Since the function $2x^2+x+2y^2+y+z^2$ is convex, its attains its maximum at one of the vertices 
\cite[Chap.~II, \S 7.1, prop.~1]{Bour}; its values there are $n^2, 2n^2-3n+2, 2n^2-3n+2$.
Since $n^2 \leqslant 2n^2-3n+2$ for $n \geqslant 2$, this shows that  $ \dim Z_G(g) \leqslant 
2n^2-3n+2$, hence $d(g) \geqslant 4n-2$, and {\it a fortiori}  $d(g) \geqslant 4n-4$, as wanted.

\medskip

[The proof also shows that $d(g) = 4n-4$ is only possible when $g$  is an involution of type
$\pm \ \iota(-1)$, as in Remark 4.]
   
   \bigskip
   
   \n {\it Proof of  $d(g) \geqslant 4n-2$ when $\Tr(g)=0$.}
   
   \smallskip
   
   We use the same notation as in the above proof. The case $z=0$ is possible only if $n$ is even, with $x=y = n/2$. This gives a centralizer of dimension $n^2+n$, hence $d(g) = n^2$, which is $> 4n-2$ when $n \geqslant 4$. Hence $z \geqslant 1$, in which case the computation given above shows that $d(g) \geqslant 4n-2$.

   \bigskip
   
   \bigskip
   
   Best wishes
   
   \bigskip
   
\n    J-P. Serre

\section{Letter by J-P.~Serre on the continuity of the density function}
\label{SerreApp2}

\def\GSp{\mathop{\mathbf{GSp}}}
\def\Sp{\mathop{\mathbf{Sp}}}
\def\GL{\mathop{\mathbf{GL}}}
\def\SL{\mathop{\mathbf{SL}}}
\def\Z{\mathop{\mathbf{Z}}}
\def\R{\mathop{\mathbf{R}}}
\def\C{\mathop{\mathbf{C}}}
\def\USp{\mathop{\mathbf{USp}}}
\def\F{\mathop{\mathbf{F}}}

\def\eps{{\varepsilon}}
\def\l{{\lambda}}
\def\a{{\mathfrak a}}
\def\b{{\mathfrak b}}
\def\n{\noindent}
\def\D{\Delta}
\def\G{\Gamma}

Paris, April 12, 2015

\bigskip
  Dear Katz,
  
  \bigskip
  
  Thank you very much for your letter about the density problem for  $\USp(2n)$.
  
    \bigskip

After writing to you I found a different way of getting the same result, based on an integration formula which is a combination of H.Weyl's formula and a formula of Steinberg \cite[Lemma 8.2]{St}.
\smallskip
 
 Let me start with a simple simply connected group  $G$. Let $T$ be a maximal torus, and define the roots, weights, fundamental weights as usual. I need to number the fundamental weights:
 $\omega_1,..., \omega_n$, where  $n$ is the rank. Call $\chi_i$ the traces of the corresponding
 fundamental representations, and call $\psi_i$ their restrictions to $T$ (I am copying Steinberg's notations). Steinberg's formula is a formula relating the $n$-differential forms on $T$ given, on one hand by the exterior product of the $d\psi_i$, on the other hand by the exterior product of the $d\omega_i/\omega_i$ [invariant differential on the torus]. The formula is :
   
   $$ (1) \quad d\psi_1 \wedge ... \wedge d\psi_n = f. d\omega_1/\omega_1 \wedge\cdots \wedge d\omega_n/\omega_n, $$
   
 \n   where $f = \omega_0 \prod_{\alpha > 0} (1 - \alpha^{-1})$ and $\omega_0 = \prod  \omega_i$.
 
\n  [Note that, here, I am forced to use a multiplicative notation for the roots, since I view them as functions on  $T$.] 
 
 We may write  $f^2$ in a slightly simpler form:
 
 $$ (2) \quad f^2 = \prod_{\alpha > 0} (\alpha + \alpha^{-1} - 2).$$
 
 This shows that  $f^2$ is real and invariant by the Weyl group. It can thus be written as a polynomial
 in the $\chi_i$ ; let me call $D$ that polynomial (it is a kind of discriminant: it vanishes only on the singular elements of $G$). We thus have:
 
  $$ (3) \quad f^2 = D(\chi_1,...,\chi_n).$$
 
   This formula of Steinberg gives 
   an integration formula over any local field. Here I shall stick to $\R$ but I have no doubt that the $p$-adic case should be useful, too. To simplify matters, I shall suppose that -1 is in the Weyl group  $W$.
   
   Let now call $UG$ the compact form of $G$, and $UT$ the corresponding torus. The roots
   $\alpha$, and the characters $\omega_i$ are now viewed as functions  on $UT$ with complex values of absolute value 1. The $\chi_i$ are real valued functions (because of my assumption on the Weyl group); let me call them  $x_i$; they give a map $x: UG \to \R^n$ which is well-known (since Elie Cartan \cite[pp.~803--804]{Ca}) to have the following properties:
   
    a) It gives a homeomorphism of the space  $Cl(UG)$ onto a compact subset $C$ of $\R^n$.
    
    [When $G$ has type $G_2$, the set $C$ is the one I asked you to draw for me.]

    b) Let $C_T$ be the standard fundamental domain of $W$ (in the tangent space, it corresponds to the fundamental alcove); the map  $C_T \to C$ is a homeomorphism; the boundary of $C$ 
    corresponds to the singular classes. [I see that by using topological arguments.]
    
    c) The function $D(x_1,...,x_n)$ [where  $D$  is as above] is a polynomial whose restriction to $C$ is zero on the boundary and nowhere else.  
    
    \smallskip
    
    By combining this with H. Weyl's integration formula, one finds:
    
     \smallskip
    
    Theorem - {\it The image by  $x : UG \to \R^n$ of the normalized Haar measure of  $G$ has a continuous density} [with respect to the standard measure $dx_1\cdots dx_n$], {\it namely the function  $\varphi(x_1,...,x_n)$ which is equal to  $0$ outside $C$, and to   $ (2\pi)^{-n} |D(x_1,...,x_n)|^{1/2}$  on $C$.}
    
     \smallskip 
    
    Corollary - {\it The equidistribution measure associated with a fundamental character of $G$ has a continuous density.} 
    
    \smallskip
    
    More precisely, the density at a number  $c$  of the fundamental character $\chi_1$ is equal to  $\int \varphi(c,x_2,...,x_n) dx_2 \cdots dx_n.$
    
     \smallskip
    
    Curiously, this point of view does not seem to give the fact that such densities are real analytic outside a finite number of values (namely, those taken by the character at the points of finite order
    of  $G$  corresponding to the vertices of the alcove, i.e. the points of $G$ of order 1 or 2
    when $G = \Sp_{2n}$).
    
     \smallskip
   
 One can also say when the density is  not 0; for instance, for the trace when $G= \Sp_{2n}$
 the density is nonzero when the trace  $c$  is such that  $-2n < c < 2n$.
 
 \medskip
 
 When -1  is not in the Weyl group, some fundamental characters come in pairs of conjugate ones,
 and instead of  $\R^n$ one should take a product of copies of  $\R$ and $
 \C$. The case of $\SL_3$
 is especially nice; the compact  $C$ lies inside $\C$, and is the interior (+ boundary) of a ``hypocycloid with 3 cusps" (hypocyclo•de ˆ trois rebroussements - as I learned when preparing the ENS competition in 1944-45).

 \bigskip
 
 Best wishes
 
  \bigskip

 J-P. Serre
 
  \bigskip

  \n  PS - The explicit formula for  $\varphi$ in the case of  $\Sp_4$ is given in the paper of Fit\'{e}, Kedlaya and others \cite{FiKeRoSu}; see Table 5, last line. Note that their  $a_1$ is my $x_1$ and their $a_2$
is my $x_2+1$.

\end{document}